\documentclass[journal]{IEEEtran}

\IEEEoverridecommandlockouts                              
\usepackage{graphics} 
\usepackage{amsmath} 
\usepackage{amssymb}  
\usepackage{mathrsfs}
\usepackage{enumerate}
\usepackage{pifont}
\usepackage{epsfig}
\usepackage{psfrag}

\newtheorem{theorem}		{Theorem}[section]

\newtheorem{lemma}		[theorem]{Lemma}
\newtheorem{sideremark}		[theorem]{Remark}
\newtheorem{sidenote}		[theorem]{Note}
\newtheorem{sideeg}		[theorem]{Example}
\newtheorem{sideconj}		[theorem]{Conjecture}
\newtheorem{sideassumption}	[theorem]{Assumption}

\newenvironment{remark}		{\begin{sideremark}\rm}{\end{sideremark}}

\newenvironment{assumption}	{\begin{sideassumption}\it}{\end{sideassumption}}

\allowdisplaybreaks

\newcommand{\done}{\ding{182}}
\newcommand{\dtwo}{\ding{183}}
\newcommand{\dthree}{\ding{184}}
\newcommand{\dfour}{\ding{185}}
\newcommand{\dfive}{\ding{186}}
\newcommand{\dsix}{\ding{187}}

\newcommand{\ba}			{\begin{array}}
\newcommand{\ea}			{\end{array}}
\newcommand{\nn}			{{\nonumber}}

\newcommand{\er}[1]		{{(\ref{#1})}}

\newcommand{\op}[1]		{{\mathcal{#1}}}
\newcommand{\ophat}[1]		{{\widehat{\op{#1}}}}
\newcommand{\ts}[1]			{{\textstyle{#1}}}

\newcommand{\pdtone}[2]		{{\frac{\partial {#1}}{\partial {#2}}}}

\newcommand{\cQ}			{{\mathbb{Q}}}
\newcommand{\cSym}		{{\mathbb{S}}}

\newcommand{\cY}			{{\mathscr{Y}}}
\newcommand{\demi}		{{\ts{\frac{1}{2}}}}
\newcommand{\dom}		{{\mathsf{dom}\, }}

\newcommand{\Hinfty}		{{\mathscr{H}_\infty}}
\newcommand{\hinfty}		{{$\Hinfty$}}
\newcommand{\Htwo}		{{\mathscr{H}_2}}
\newcommand{\htwo}		{{$\Htwo$}}
\newcommand{\Ltwo}		{{\mathscr{L}_2}}
\newcommand{\ltwo}			{{$\Ltwo$}}
\newcommand{\N}			{{\mathbb{N}}}
\newcommand{\R}			{{\mathbb{R}}}

\begin{document}


\title{A new fundamental solution for a class of differential Riccati equations$^*$
}

\author{Peter M. Dower$^{\dagger}$ \qquad\qquad\qquad Huan Zhang$^{\dagger}$
\thanks{*This research was partially supported by the Australian Research Council and AFOSR/AOARD grants DP120101549 and FA2386-12-1-4084 respectively.}
\thanks{$^{\dagger}$Dower and Zhang are with the Department of Electrical \& Electronic Engineering at the University of Melbourne, 
        Victoria 3010, Australia.
        {\tt\small \{pdower,hzhang\}@unimelb.edu.au}}%
}


\maketitle
\thispagestyle{empty}
\pagestyle{empty}


\begin{abstract}
A class of differential Riccati equations (DREs) is considered whereby the evolution of any solution can be identified with the propagation of a value function of a corresponding optimal control problem arising in {\ltwo}-gain analysis. By exploiting the semigroup properties inherited from the attendant dynamic programming principle, a max-plus primal space fundamental solution semigroup of max-plus linear max-plus integral operators is developed that encapsulates all such value function propagations. Using this semigroup, a new one-parameter fundamental solution semigroup of matrices is developed for the aforementioned class of DREs. It is demonstrated that this new semigroup can be used to compute particular solutions of these DREs, and to characterize finite escape times (should they exist) in a relatively simple way compared with that provided by the standard symplectic fundamental solution semigroup.
\end{abstract}


\section{INTRODUCTION}
\label{sec:intro}

Differential Riccati equations (DREs) arise naturally in linear optimal control and dissipative systems theory \cite{AM:71,DGKF:89,PAJ:91,GL:95}. A typical finite dimensional DRE applicable in the verification of the {\ltwo}-gain property for linear systems is an ordinary differential equation defined via matrices $A\in\R^{n\times n}$, $B\in\R^{n\times m}$, $C\in\R^{p\times n}$, $n,m,p\in\N$, by
\begin{align}
	\dot P_t
	& = A' P_t + P_t A + P_t B B' P_t + C' C\,,
	\label{eq:DRE}
\end{align}
in which $P_t\in\cSym^{n\times n}$ describes a particular symmetric matrix valued solution evolved forward from an initial condition
\begin{align}
	P_0 & \in\cSym_{>M}^{n\times n},
	\label{eq:IC-M}
\end{align}
residing in the space of symmetric matrices exceeding some $M\in\cSym^{n\times n}$, to any time $t\in[0,t^*)$ in some maximal horizon of existence $t^* = t^*(P_0)\in\R_{>0}^+ \doteq \R_{>0}\cup\{+\infty\}$. Related DREs arise in linear {\htwo}- and {\hinfty}-control and filtering, etc, see for example \cite{DGKF:89,PAJ:91,GL:95}.

A fundamental solution for DRE \er{eq:DRE} is a mathematical object that characterizes every possible solution of that DRE, as parameterized by its initial (or terminal) condition \er{eq:IC-M}. One such fundamental solution is the {\em symplectic fundamental solution}, which is itself the solution of a (derived) Hamiltonian system of linear ordinary differential equations, see for example \cite{AM:71,DM:73,LL:06}. Another fundamental solution is the {\em max-plus dual-space fundamental solution} \cite{M:08,DM1:15,ZD1:15,ZD2:15}, which is constructed by exploiting semiconvex duality \cite{FM:00} and max-plus linearity of the Lax-Oleinik semigroup \cite{KM:97} of dynamic programming evolution operators for an associated optimal control problem, see also \cite{KM:97,LMS:01,AGK:05,CGQ:04,M:06}.

In this paper, a new {\em max-plus primal space fundamental solution} is provided for DREs of the form \er{eq:DRE}, \er{eq:IC-M}. This fundamental solution can be used to evaluate particular solutions of \er{eq:DRE}, analogously to the symplectic and max-plus dual space fundamental solutions. Its development is complementary to that of the max-plus dual space fundamental solution documented in \cite{M:08,DM1:15,ZD2:15}, and parallels the corresponding recent primal space development for difference Riccati equations \cite{ZD1:15}. It is shown that this new fundamental solution provides a simpler test for establishing existence of solutions of \er{eq:DRE}, \er{eq:IC-M} when compared with the symplectic fundamental solution.

In terms of organization, the symplectic fundamental solution for DRE \er{eq:DRE} is recalled in Section \ref{sec:symplectic} for comparative purposes, to formalize existence of solutions, and to construct a specific particular solution to \er{eq:DRE} of utility later. The max-plus primal space fundamental solution, and corresponding fundamental solution semigroup, is subsequently constructed in Sections \ref{sec:fund} and \ref{sec:fund-semigroups}, using the aforementioned particular solution. An illustration of its application is provided in Section \ref{sec:example}, followed by some brief concluding remarks in Section \ref{sec:conc}. Proofs are largely delayed to the appendices.

Throughout, $\N$, $\cQ$, $\R$ denote respectively the natural, rational, and real numbers, while $\R_{\ge 0}$, $\R^n$, $\R^{n\times n}$ denote respectively the nonnegative real numbers, $n$-dimensional Euclidean space, and the space of $n\times n$ matrices with real entries. $\R^\pm$, etc, denotes the analogous sets defined with respect to extended reals $\R\cup\{\pm\infty\}$. Similarly, $\cSym^{n\times n}$, $\cSym_{\ge 0}^{n\times n}$, $\cSym_{>0}^{n\times n}$ denote the spaces of symmetric, nonnegative symmetric, and positive definite symmetric elements of $\R^{n\times n}$ respectively. Further extending this notation, $\cSym_{>M}^{n\times n}$ denotes the subset of $\cSym^{n\times n}$ of matrices $P$ satisfying $P - M\in\cSym_{>0}^{n\times n}$, etc. The transpose of $P\in\R^{n\times n}$ is denoted by $P'\in\R^{n\times n}$. The corresponding identity is denoted by $I\in\cSym^{n\times n}$. Given $U\in\R^{2n\times 2n}$, the two-by-two block matrix representation
\begin{align}
	U & = \left[ \ba{cc} 
			U^{11} & U^{12} \\
			U^{21} & U^{22}
	\ea \right] \in \R^{2n\times 2n},
	\label{eq:block}
\end{align}
with $U^{ij}\in\R^{n\times n}$, $i,j\in\{1,2\}$, is used where convenient.


\section{SYMPLECTIC FUNDAMENTAL SOLUTION}
\label{sec:symplectic}
Existence of a unique solution to DRE \er{eq:DRE}, subject to \er{eq:IC-M}, may be verified by application of Banach's fixed point theorem, see for example \cite[Theorem 2.4]{DM1:15}. Alternatively, it may be constructed directly as
\begin{align}
	P_t
	& = Y_t X_t^{-1}
	\label{eq:symplectic-P}
\end{align}
in which $X_t,Y_t\in\R^{n\times n}$ are defined with respect to the {\em symplectic fundamental solution} $\Sigma_t\in\R^{2n\times 2n}$ for \er{eq:DRE} by
\begin{align}
	\begin{gathered}
	\left[ \ba{c}
		X_t \\ Y_t
	\ea \right] = \Sigma_t \left[ \ba{c} I \\ P_0 \ea \right],
	\quad
	t\in[0,t^*(P_0))\,,
	\\
	\Sigma_t \doteq \exp(\mathcal{H} t)\,,
	\quad
	\mathcal{H}
	\doteq \left[ \ba{cc}
		-A & -B B' \\
		\ C' C & A'
	\ea \right],
	\end{gathered}
	\label{eq:symplectic-XY}
\end{align}
in which the maximal horizon of existence $t^*(P_0)\in\R_{>0}^+$ of the unique particular solution $P_t$ in \er{eq:symplectic-P} is characterized by
\begin{align}
	t^*(P_0)
	& \doteq
	\sup \left\{ t\in\R_{>0} \, \left| \! \ba{c}
				 X_s^{-1} \text{ exists } \forall \ s\in(0,t]
				 \\
				\text{with }X_s \text{ given by \er{eq:symplectic-XY}}
				\\
				 \text{subject to } P_0\in\cSym^{n\times n}
	\ea \! \right. \right\},
	\label{eq:symplectic-escape}
\end{align}
see \cite{DM:73,S:82,KB:12}. This maximal horizon of existence is either strictly positive and finite, or infinite. Where $t^*(P_0)$ is strictly positive, the solution $P_t$ experiences a finite escape at $t=t^*(P_0)$. Otherwise, no such such finite escape time exists, and $P_t$ may be evolved to any arbitrarily large time horizon $t\in\R_{>0}$. For example, under the conditions of the strict bounded real lemma (e.g. \cite[Theorem 2.1]{PAJ:91} or \cite[Theorem 3.7.4]{GL:95}), $P_0 = 0\in\cSym^{n\times n}$ implies that $t^*(P_0) = +\infty$.

By inspection, the symplectic fundamental solution $\Sigma_t$, defined by \er{eq:symplectic-P}, \er{eq:symplectic-XY}, \er{eq:symplectic-escape} satisfies the properties of a fundamental solution for DRE \er{eq:DRE}. In particular, it can be evolved independently of any specific DRE initial condition $P_0$, and can be used to recover any such particular solution via an operation involving that $P_0$. It is a standard tool for the representation and computation of solutions to DREs of the form \er{eq:DRE}. In Section \ref{sec:fund}, it is used to construct a particular solution of a DRE of the form \er{eq:DRE} that is employed in the construction the max-plus primal space fundamental solution of interest.


\section{MAX-PLUS FUNDAMENTAL SOLUTION}
\label{sec:fund}

\newcommand{\cSvex}[1]		{{\mathscr{S}_+^{{#1}}}}
\newcommand{\cSave}[1]		{{\mathscr{S}_-^{{#1}}}}


\subsection{Max-plus algebra and semiconvex duality}
The max-plus algebra \cite{KM:97,M:08} is a commutative semifield over $\R^-$, equipped with addition and multiplication operators defined respectively by $a\oplus b\doteq \max(a,b)$ and $a\otimes b\doteq a + b$. It is an idempotent algebra, as the $\oplus$ operation is idempotent (i.e. $a\oplus a = a$), and a semifield as additive inverses do not exist. The max-plus integral of a function $f:\R^n\rightarrow\R^-$ over a subset $\cY\subset\R^n$ of its domain is $\int_\cY^\oplus f(y)\, dy \doteq \sup_{y\in\cY} f(y)$. The max-plus delta function $\delta^-:\R^n\times\R^n\rightarrow\R^-$ is defined for all $x,y\in\R^n$ by
\begin{align}
	\delta^-(x,y)
	& \doteq \left\{ \ba{rl}
		0\,, & x=y\,,
		\\
		-\infty\,, & x\ne y\,.
	\ea \right.
	\label{eq:delta}
\end{align}
In developing a max-plus fundamental solution, it is useful to introduce spaces of uniformly semiconvex and semiconcave functions, defined with respect to $K\in\cSym^{n\times n}$, by
\begin{equation}
	\begin{aligned}
	\cSvex{K}
	& \doteq \left\{ f:\R^n\rightarrow\R^- \, \left| \, \ba{c}
									f + \demi\, \langle \cdot ,\, K\,  \cdot \rangle
									\\
									\text{convex}
								\ea \right. \right\},
								\quad
	\\
	\cSave{K}
	& \doteq \left\{ a:\R^n\rightarrow\R^- \, \left| \, \ba{c}
									a - \demi\, \langle \cdot ,\, K\,  \cdot \rangle
									\\
									\text{concave}
								\ea \right. \right\},
								\quad								
	\end{aligned}
	\label{eq:semiconvex}
\end{equation}
respectively. Semiconvex duality is a duality between these spaces of semiconvex and semiconcave functions, that is established via the semiconvex transform \cite{FM:00}. The semiconvex transform is a generalization of the Legendre-Fenchel transform \cite{M:70,R:74,RW:97}, in which convexity is weakened to semiconvexity via a quadratic basis function $\varphi:\R^n\times\R^n\rightarrow\R$. This basis function is defined for all $x,z\in\R^n$ by
\begin{align}
	& \varphi(x,z)
	\doteq \demi  (x-z)'  M (x-z) = \demi \left[ \ba{c} x \\ z \ea \right]' \mu(M)\, \left[ \ba{c} x \\ z \ea \right],
	\label{eq:basis}
\end{align}
in which $M\in\cSym^{n\times n}$, and $\mu:\cSym^{n\times n}\rightarrow\cSym^{2n\times 2n}$ is defined by
\begin{align}
	& \mu(P)
	\doteq \left[ \ba{cc}
			+P & -M \\
			-M & +M
		\ea \right]\in\cSym^{2n\times 2n},
	\label{eq:mu}
\end{align}
for all $P\in\cSym^{n\times n}$.
\begin{assumption}
\label{ass:M}
Matrix $M\in\cSym^{n\times n}$ defining the basis \er{eq:basis} satisfies the following properties:
\begin{enumerate}[1)]
\item $M^{-1}\in\cSym^{n\times n}$ exists;
\item $t^*(M)=+\infty$, cf. \er{eq:symplectic-escape}.
\end{enumerate}
\end{assumption}

Standard conditions under which Assumption \ref{ass:M} holds are controllability and observability of $(A,B)$ and $(C,A)$ respectively, or via the strict bounded real lemma, see for example \cite{PAJ:91}. The details are postponed to Lemma \ref{lem:ass-M-holds}.

The semiconvex transform and its inverse are well-defined with respect to the basis $\varphi$ of \er{eq:basis} by
\begin{align}
	\op{D}_\varphi\, \psi
	& \doteq -\int_{\R^n}^\oplus \varphi(x,\cdot) \otimes (-\psi(x)) \, dx\,,
	\label{eq:op-D}
	\\
	\op{D}_\varphi^{-1}\, a
	& \doteq \int_{\R^n}^\oplus \varphi(\cdot,z) \otimes a(z)\, dz\,,
	\label{eq:op-D-inv}
\end{align}
for all $\psi\in\dom(\op{D}_\varphi) \doteq \cSvex{-M}$ and $a\in\dom(\op{D}_\varphi^{-1}) \doteq \cSave{-M}$, see also \cite{M:06,M:08,DM1:15,DMZ1:15}. For quadratic functions, \er{eq:op-D} and \er{eq:op-D-inv} define a pair of matrix operations on corresponding spaces of Hessians. In particular, with $\psi:\R^n\rightarrow\R$ defined with respect to some $P\in\cSym_{> M}^{n\times n}$ by $\psi(x) \doteq \demi x' P x$ for all $x\in\R^n$, application of \er{eq:op-D} yields a well-defined semiconvex dual. In particular, $a(z) = \demi z' \Upsilon(P) z$ for all $z\in\R^n$, with $\Upsilon:\cSym^{n\times n}\rightarrow\cSym^{n\times n}$ defined by
\begin{align}
	& \Upsilon(P)
	\doteq -M - M (P-M)^{-1} M,
	\qquad 
	P\in\dom(\Upsilon),
	\nn\\
	& \dom(\Upsilon)
	\doteq \cSym_{>M}^{n\times n}.
	\label{eq:Upsilon}
\end{align}
Similarly, the inverse semiconvex transform \er{eq:op-D-inv} corresponds to the inverse map $\Upsilon^{-1}$, with
\begin{align}
	& \Upsilon^{-1}(P)
	\doteq M - M (P+M)^{-1} M,
	\qquad P\in\dom(\Upsilon^{-1}),
	\nn\\
	& \dom(\Upsilon^{-1}) \doteq \cSym_{<-M}^{n\times n}\,.
	\label{eq:Upsilon-inv}
\end{align}
\begin{remark}
\label{rem:pseudo}
The domains specified in \er{eq:Upsilon} and \er{eq:Upsilon-inv} may be extended to $\cSym_{\ge M}^{n\times n}$ and $\cSym_{\le -M}^{n\times n}$ respectively, via corresponding Moore-Penrose pseudo-inverses. However, this extension is not required here, and the details are omitted.
\end{remark}


\subsection{Optimal control problem}
In order to construct a max-plus fundamental solution for the propagation of solutions of DRE \er{eq:DRE}, \er{eq:IC-M}, it is useful to define a corresponding optimal control problem on a finite time horizon $t\in\R_{\ge 0}$ via the value function $W_t:\R^n\rightarrow\R$ given by
\begin{align}
	W_t(x)
	& = (\op{S}_t\, \Psi)(x)
	\label{eq:value-W}
\end{align}
for all $x\in\R^n$. Here, $\Psi:\R^n\rightarrow\R$ denotes the terminal payoff $\Psi(x) \doteq \demi x' P_0\, x$ for all $x\in\R^n$, in which $P_0\in\cSym_{>M}^{n\times n}$ is as per \er{eq:IC-M}, with $M\in\cSym^{n\times n}$ as per \er{eq:basis}. The {\em dynamic programming evolution operator} $\op{S}_t$ appearing in \er{eq:value-W} is defined by
\begin{align}
	& (\op{S}_t\, \psi)(x)
	\doteq \sup_{w\in\Ltwo([0,t];\R^m)} J_\psi(t,x,w),
	\quad\psi\in\dom(\op{S}_t),
	\nn\\
	& \dom(\op{S}_t) \doteq \left\{ \psi:\R^n\rightarrow\R^- \, \left| \ba{c}
									(\op{S}_t\, \psi)(x)\in\R^-
									\\ 
									\forall\ x\in\R^n
							\ea \right. \right\},
	\label{eq:op-S}
\end{align}
for all $x\in\R^n$. Payoff $J_\psi(t,\cdot,\cdot):\R^n\times\Ltwo([0,t];\R^m)\rightarrow\R^-$ is defined by
\begin{align}
	& J_\psi(t,x,w)
	\doteq \int_0^t \demi |y_s|^2 - \demi |w_s|^2 \, ds + \psi(x_t)
	\label{eq:payoff}
\end{align}
for all $x\in\R^n$, $w\in\Ltwo([0,t];\R^m)$, in which $x_s\in\R^n$, $w_s\in\R^m$, and $y_s\in\R^p$ denote the state, input, and output (respectively) of the linear system 
\begin{equation}
	\begin{aligned}
	\dot x_s
	& = A x_s + B w_s\,,
	&& x_0 = x\in\R^n\,,
	\\
	y_s & = C x_s\,,
	&&
	\end{aligned}
	\label{eq:dynamics}
\end{equation}
at time $s\in[0,t]$.  It is straightforward to show that the value function $W_t$ of \er{eq:value-W} is quadratic, see \cite{AM:71,M:08,DM1:15,ZD1:15}, with
\begin{align}
	W_t(x)
	& = (\op{S}_t \Psi)(x) = \demi x' P_t x 
	\label{eq:W}
\end{align}
for all $x\in\R^n$, with $P_t\in\cSym^{2n\times 2n}$ satisfying DRE \er{eq:DRE} subject to the initial condition \er{eq:IC-M}.


\subsection{Auxiliary optimal control problem}
It constructing a max-plus fundamental solution for \er{eq:DRE}, it is useful to introduce an auxiliary optimal control problem defined on the same finite time horizon $t\in\R_{\ge 0}$ with value function $S_t(\cdot,z):\R^n\rightarrow\R$, $z\in\R^n$, defined in terms of the dynamic programming evolution operator $\op{S}_t$ of \er{eq:op-S} by
\begin{align}
	S_t(x,z)
	& \doteq (\op{S}_t \, \varphi(\cdot,z))(x)
	\label{eq:value-S}
\end{align}
for all $x\in\R^n$. This value function is again quadratic, with
\begin{align}
	S_t(x,z)
	&
	= \demi\, \left[ 
			\ba{c} x \\ z \ea \right]' Q_t\, \left[ \ba{c} x \\ z \ea \right],
	\label{eq:S}
\end{align}
for all $x,z\in\R^n$, in which $Q_t\in\cSym^{2n\times 2n}$ is the unique solution of the DRE
\begin{align}
	\dot Q_t
	& = \hat A' Q_t + Q_t \hat A + Q_t\, \hat B \hat B' Q_t + \hat C' \hat C
	\label{eq:Q-DRE}
\end{align}
initialized with
\begin{align}
	& Q_0 = \mu(M)\in\R^{2n\times 2n}
	\label{eq:Q-IC}
\end{align}
as per \er{eq:basis}, \er{eq:mu}, for all $t\in[0,t^*(Q_0))$. Here, $t^*(Q_0)\in\R_{>0}^+$ denotes the corresponding maximal horizon of existence \er{eq:symplectic-escape}, while the constant matrices $\hat A\in\R^{2n\times 2n}$, $\hat B\in\R^{2n\times m}$, and $\hat C\in\R^{p\times 2n}$ appearing in \er{eq:Q-DRE} are defined by
\begin{align}
	\hat A
	& \doteq \left[ \ba{cc}
			A & 0 \\
			0 & 0
	\ea \right],
	\quad
	\hat B \doteq \left[ \ba{c}
				B \\
				0
	\ea \right],
	\quad
	\hat C \doteq \left[ \ba{cc}
				C & 0
	\ea \right].
	\label{eq:ABC-hat}
\end{align}
Equivalently, using the notation of \er{eq:block}, DRE \er{eq:Q-DRE}, \er{eq:Q-IC} implies that $Q_t^{11}, Q_t^{22}\in\cSym^{n\times n}$, $Q_t^{12}\in\R^{n\times n}$ satisfy
\begin{align}
	\dot Q_t^{11}
	& = A' Q_t^{11} + Q_t^{11} A + Q_t^{11} B B' Q_t^{11} + C' C\,,
	\label{eq:Q-11}
	\\
	\dot Q_t^{12}
	& = (A + B B' Q_t^{11})' Q_t^{12}\,, \ Q_t^{21} = (Q_t^{12})'\,,
	\label{eq:Q-12}
	\\
	\dot Q_t^{22}
	& = (Q_t^{12})' B B' Q_t^{12}\,,
	\label{eq:Q-22}
\end{align}
for all $t\in[0,t^*(M))$, subject to $Q_0^{11} = -Q_0^{12} = Q_0^{22} = M$, with $M\in\cSym^{n\times n}$ as per \er{eq:basis}. As \er{eq:Q-12} and \er{eq:Q-22} describe (respectively) a linear evolution equation and an integration, any finite escape of $Q_t$ must be due to the dynamics \er{eq:Q-11}, see for example \cite[Proposition 3.6(iv)]{BDDM:07}. That is, the maximal horizon of existence for \er{eq:Q-DRE} and \er{eq:Q-11} must be equal, ie. $t^*(Q_0) = t^*(M)$. Assumption \ref{ass:M} further implies that
\begin{align}
	& t^*(Q_0) = t^*(M) = +\infty.
	\label{eq:equal-escape}
\end{align}
As DRE \er{eq:Q-11} is of the same form as \er{eq:DRE}, the particular solution $Q_t$ of DRE \er{eq:Q-DRE}, \er{eq:Q-IC} can be characterized explicitly via the symplectic fundamental solution \er{eq:symplectic-XY}.
\begin{theorem}
\label{thm:symplectic-and-Q}
Under Assumption \ref{ass:M}, the particular solution $Q_t$ of DRE \er{eq:Q-DRE}, \er{eq:Q-IC} and the symplectic fundamental solution $\Sigma_t$ of \er{eq:symplectic-XY} for DRE \er{eq:DRE} are equivalent. That is, there exists an invertible operator $\Xi:\cSym^{2n\times 2n}\rightarrow\cSym^{2n\times 2n}$ such that
\begin{align}
	Q_t
	& = \Xi(\Sigma_t),
	\qquad
	\Sigma_t = \Xi^{-1} (Q_t)
	\label{eq:symplectic-and-Q}
\end{align}
for all $t\in\R_{\ge 0}$.
\end{theorem}
\begin{proof}
See Appendix \ref{app:symplectic-and-Q}.
\end{proof}

Theorem \ref{thm:symplectic-and-Q} demonstrates that under the conditions of Assumption \ref{ass:M}, any particular solution of the DRE \er{eq:DRE}, \er{eq:IC-M} can be represented equivalently by the symplectic fundamental solution $\Sigma_t$ of \er{eq:symplectic-XY}, or via the Hessian $Q_t$ of the quadratic value function of the auxiliary optimal control problem \er{eq:value-S}, \er{eq:S}, see \er{eq:symplectic-and-Q}.
Consequently, the following sufficient condition for Assumption \ref{ass:M} is useful.

\begin{lemma}
\label{lem:ass-M-holds}
Suppose there exists a stabilizing solution $M_0\in\cSym_{\ge 0}$ of the {\em algebraic Riccati equation} (ARE)
\begin{align}
	0 & = A' M_0 + M_0 A + M_0 B B' M_0 + C'C.
	\label{eq:ARE}
\end{align}
Then, there always exists an invertible $M\in\cSym^{n\times n}$ satisfying 
\begin{align}
	& M - M_0\in\cSym_{<0}^{n\times n},
	\label{eq:ass-M-holds}
\end{align}
such that Assumption \ref{ass:M} holds.
\end{lemma}
\begin{proof}
See Appendix \ref{app:ass-M-holds}.
\end{proof}

\begin{remark}
Lemma \ref{lem:ass-M-holds} provides a constructive approach to validating Assumption \ref{ass:M} directly. It also enables indirect validation via the bounded and strict bounded real lemmas, see for example \cite{PAJ:91}.
In particular, stability of $A$, controllability of $(A,B)$, observability of $(C,A)$, and the finite gain property $\|(A,B,C)\|_{\Hinfty}\le1$ imply via the bounded real lemma that Assumption \ref{ass:M} holds. Alternatively, stability of $A$ and the strict gain property $\|(A,B,C)\|_{\Hinfty}<1$ imply via the strict bounded real lemma that Assumption \ref{ass:M} holds. 
\end{remark}


\newcommand{\cP}		{{\mathscr{P}}}


\subsection{Max-plus integral operator representations for \er{eq:op-S}}
A horizon indexed {\em max-plus linear max-plus integral operator} defined on a space $\cP$ is an operator of the form
\begin{align}
	& \op{F}_t^\oplus \, \pi
	\doteq \int_{\cP}^\oplus F_t(\cdot,\omega)\otimes \pi\circ\chi_t(\cdot, \omega)\, d\omega\,,
	\quad \pi\in\dom(\op{F}_t^\oplus),
	\nn\\
	& \dom(\op{F}_t^\oplus) \doteq \left\{ \pi:\cP\rightarrow\R^- \, \left| \ba{c}
											( \op{F}_t^\oplus \pi)(x)\in\R^-
											\\
											\forall \ x\in\R^n
										\ea \right. \right\},
	\label{eq:op-F}
\end{align}
where $F_t:\R^n\times\cP\rightarrow\R^-$ denotes the kernel of the operator, $\chi_t:\R^n\times\cP\rightarrow\R^n$ is an auxiliary operator (included here for generality), and $\pi\in\dom(\op{F}_t^\oplus)$ is the function-valued argument of $\op{F}_t^\oplus$ representing a terminal payoff (or value function) or its semiconvex dual. The dynamic programming evolution operator $\op{S}_t$ of \er{eq:op-S} defines a max-plus linear max-plus integral operator of this form, with 
\begin{align}
	& \cP \doteq \Ltwo([0,t];\R^m)\,,
	\nn\\
	& F_t(x,w) = I_t(x,w) \doteq \int_0^t \demi |y_s|^2 - \demi |w_s|^2 \, ds\,,
	\nn\\
	&
	\chi_t(x,w) \doteq x_t\,, 
	\nn
\end{align}
where  $I_t(x,w)$ is the integrated running payoff associated with initial state $x\in\R^n$ and input $w\in\Ltwo([0,t];\R^m)$ over the horizon $t\in\R_{\ge 0}$, and $x_t\in\R^n$ is the corresponding terminal state, both defined with respect to \er{eq:dynamics}. That is, for all $x\in\R^n$,
\begin{align}
	(\op{S}_t\, \psi)(x)
	& = \int_{\Ltwo([0,t];\R^m)}^\oplus I_t(x,w) \otimes \psi(x_t)\, dw\,.
	\label{eq:op-S-integral}
\end{align}
Similarly, recalling the definition \er{eq:delta} of the max-plus delta function $\delta^-$, the identity max-plus linear max-plus integral operator on $\cP\doteq\R^n$, defined via $\chi_t(x,y) \doteq y\in\R^n$, is
\begin{equation}
	(\op{I}^\oplus\, \psi)(x) \doteq \int_{\R^n}^\oplus \delta^-(x,y)\otimes \psi(y)\, dy\,,
	\label{eq:op-I}
\end{equation}
for all $x\in\R^n$, ie. $\op{I}^\oplus\psi = \psi$ for any $\psi\in\dom(\op{I}^\oplus)$, in which the domain $\dom(\op{I}^\oplus)$ is defined as per \er{eq:op-F}.
\begin{theorem}
\label{thm:op-G}
Under Assumption \ref{ass:M}, and given the dynamic programming evolution operator $\op{S}_t$ of \er{eq:op-S} with $t\in\R_{\ge 0}$ fixed, there exists a max-plus linear max-plus integral operator $\op{G}_t^\oplus$ of the form \er{eq:op-F} such that 
\begin{align}
	& \op{S}_t\, \psi = \op{G}_t^\oplus\, \psi
	\doteq \int_{\R^n}^\oplus G_t(\cdot,y)\otimes \psi(y)\, dy,
	\quad
	\forall \ \psi\in\dom(\op{G}_t^\oplus),
	\nn\\
	& \dom(\op{G}_t^\oplus) \doteq \dom(\op{S}_t),
	\label{eq:op-G}
\end{align}
with kernel $G_t:\R^n\times\R^n\rightarrow\R^-$ defined for all $x,y\in\R^n$ by
\begin{align} 
	& G_t(x,y) \doteq (\op{S}_t\, \delta^-(\cdot,y))(x) = (\op{D}_\varphi S_t(x,\cdot))(y),
	\label{eq:kernel-G-def}
\end{align}
with respect to \er{eq:delta}, \er{eq:op-D}, \er{eq:op-S}, \er{eq:value-S}.
\end{theorem}
\vspace{2mm}
\begin{proof}
Fix arbitrary $t\in\R_{\ge 0}$ and $x,y\in\R^n$. Recalling the definition \er{eq:value-S}, \er{eq:S} of $S_t$,
\begin{align}
	& S_t(x,y)
	= (\op{S}_t\, \varphi(\cdot,y))(x) 
	= \demi x' Q_t^{11} x + x' Q_t^{12} y + \demi y' Q_t^{22} y,
	\nn
\end{align}
wherein Assumption \ref{ass:M} and \er{eq:Q-22} imply that
\begin{align}
	& Q_t^{22}\in\cSym_{\ge M}^{n\times n}.
	\label{eq:Q-22-ge-M}
\end{align}
Consequently, by definition \er{eq:op-D} of the semiconvex transform, $S_t(x,\cdot)\in\dom(\op{D}_\varphi) = \cSvex{-M}$, so that
\begin{align}
	& G_t(x,\cdot) \doteq \op{D}_\varphi S_t(x,\cdot)\in\cSave{-M} = \dom(\op{D}_\varphi^{-1})
	\label{eq:dual-S}
\end{align}
is well-defined. Note in particular that $G_t(x,y)\in\R^-$ by definition \er{eq:semiconvex} of $\cSave{-M}$. As $t\in\R_{\ge 0}$ and $x,y\in\R^n$ are arbitrary, a max-plus linear max-plus integral operator $\op{G}_t^\oplus$ of the form \er{eq:op-F} is well-defined by the kernel $G_t$ of \er{eq:dual-S}. Recalling the definitions \er{eq:op-D-inv}, \er{eq:value-S}, \er{eq:op-I} of $\op{D}_\varphi^{-1}$, $S_t$, $\op{I}^\oplus$,
\begin{align}
	& S_t(x,y) = (\op{S}_t\, \varphi(\cdot,y))(x)
	= (\op{S}_t\, \op{I}^\oplus\, \varphi(\cdot,y))(x)
	\nn\\
	& = \int_{\Ltwo([0,t];\R^m)}^\oplus I_t(x,w) \otimes \left[ \int_{\R^n}^\oplus \delta^-(x_t,\xi)\otimes \varphi(\xi,y) \, d\xi \right] dw
	\nn\\
	& = \int_{\R^n}^\oplus \left[ \int_{\Ltwo([0,t];\R^m)}^\oplus I_t(x,w) \otimes  \delta^-(x_t,\xi)\, dw \right] \otimes \varphi(\xi,y) \, d\xi
	\nn\\
	& = \int_{\R^n}^\oplus (\op{S}_t\, \delta^-(\cdot,\xi))(x) \otimes \varphi(\xi,y)\, d\xi
	\nn\\
	& =  \int_{\R^n}^\oplus \varphi(y,\xi) \otimes (\op{S}_t\, \delta^-(\cdot,\xi))(x)\, d\xi 
	\nn\\
	& = (\op{D}_\varphi^{-1} T_t(x,\cdot))(y)
	\label{eq:S-and-T}
\end{align}
where the interchange of max-plus integrals involved corresponds to an interchange of suprema, the second last equality follows by symmetry of $\varphi$, ie. $\varphi(\xi,y) = \varphi(y,\xi)$, and $T_t(x,y) \doteq (\op{S}_t\, \delta^-(\cdot,y))(x)$. Hence, substituting \er{eq:S-and-T} in \er{eq:dual-S},
\begin{align}
	& G_t(x,\cdot)
	= \op{D}_\varphi S_t(x,\cdot) = \op{D}_\varphi \op{D}_\varphi^{-1} T_t(x,\cdot) = T_t(x,\cdot). 
	\nn
\end{align}
That is, \er{eq:kernel-G-def} holds. Furthermore, for any $\psi\in\dom(\op{S}_t)$, a similar argument yields
\begin{align}
	& (\op{S}_t\, \psi)(x)
	= (\op{S}_t\, \op{I}^\oplus\, \psi)(x)
	\nn\\
	& = \int_{\Ltwo([0,t];\R^m)}^\oplus I_t(x,w) \otimes \left[ \int_{\R^n}^\oplus \delta^-(x_t,y)\otimes \psi(y) \, dy \right] dw
	\nn\\
	& = \int_{\R^n}^\oplus \left[ \int_{\Ltwo([0,t];\R^m)}^\oplus I_t(x,w) \otimes  \delta^-(x_t,y)\, dw \right] \otimes \psi(y) \, dy
	\nn\\
	& = \int_{\R^n}^\oplus (\op{S}_t\, \delta^-(\cdot,y))(x) \otimes \psi(y)\, dy = \int_{\R^n}^\oplus G_t(x,y) \otimes \psi(y)\, dy
	\nn\\
	& = (\op{G}_t^\oplus\, \psi)(x)\,.
	\nn
\end{align}
That is, \er{eq:op-G} holds.
\end{proof}

\begin{remark}
\label{rem:another-kernel-bound}
The kernel $G_t$ of the max-plus linear max-plus integral operator $\op{G}_t^\oplus$ defined in Theorem \ref{thm:op-G} can be bounded above by the value function of a third optimal control problem. In particular, applying \er{eq:kernel-G-def},
\begin{align}
	& G_t(x,y)
	= (\op{S}_t\, \delta^-(\cdot,y)) \le (\op{S}_t\, \psi_0)(x)
	\nn
\end{align}
for all $t\in\R_{\ge 0}$, $x,y\in\R^n$, where $\psi_0:\R^n\rightarrow\R$ is the zero terminal payoff defined by $\psi_0(x) \doteq 0$ for all $x\in\R^n$. By inspection of \er{eq:op-S}, $\op{S}_t\, \psi_0$ is the value function of a standard optimal control problem arising in {\ltwo}-gain analysis. It is finite valued if there exists a stabilizing solution of ARE \er{eq:ARE}. 
\end{remark}

In developing a max-plus fundamental solution for DRE \er{eq:DRE}, \er{eq:IC-M} via Theorem \ref{thm:op-G}, it is useful to establish a connection between finiteness of the kernel $G_t$ of \er{eq:kernel-G-def} and controllability of the underlying dynamics \er{eq:dynamics}.

\begin{assumption}
\label{ass:controllable}
$(A,B)$ of \er{eq:dynamics} is controllable.
\end{assumption}

\begin{lemma}
\label{lem:kernel-G-finite}
Suppose Assumption \ref{ass:M} holds. Then, the kernel $G_t$ of the max-plus linear max-plus integral operator $\op{G}_t^\oplus$ defined by \er{eq:op-G} satisfies the following property:
\begin{align}
	& \ba{c}
	\text{Assumption \ref{ass:controllable} holds}
	\ea
	\
	\Longleftrightarrow
	\
	\ba{c}
	G_t(x,y)\in\R
	\\
	\forall\ t\in\R_{>0}, \, x,y\in\R^n
	\ea
	\nn
\end{align}
\end{lemma}
\begin{proof}
See Appendix \ref{app:kernel-G-finite}.
\end{proof}


\subsection{Max-plus fundamental solution for DRE \er{eq:DRE}}
\label{ssec:DRE-fund}

Dynamic programming implies that the set of dynamic programming evolution operators $\{\op{S}_t\}_{t\in\R_{\ge 0}}$ defines the well-known Lax-Oleinik dynamic programming semigroup \cite{KM:97}. Applying Theorem \ref{thm:op-G}, it immediately follows that $\{\op{G}_t^\oplus\}_{t\in\R_{\ge 0}}$ must also define a one-parameter semigroup of operators via \er{eq:op-G}. In particular, $\{\op{G}_t^\oplus\}_{t\in\R_{\ge 0}}$ naturally inherits (from the Lax-Oleinik semigroup) the semigroup and identity properties
\begin{equation}
	\begin{gathered}
	\op{G}_t^\oplus\, \op{G}_\tau^\oplus = \op{G}_{t+\tau}^\oplus\,,
	\quad \op{G}_0^\oplus = \op{I}^\oplus\,,
	\end{gathered}
	\label{eq:op-G-semigroup}
\end{equation}
for $t,\tau\in\R_{\ge 0}$. This particular semigroup is referred to as the {\em max-plus primal space fundamental solution semigroup} for the optimal control problem \er{eq:value-W}, see \cite{ZD1:15,DMZ1:15}. The modifier {\em primal} used here refers to the fact that propagation occurs in the primal space of payoffs. (A corresponding {\em max-plus dual space fundamental solution semigroup} also exists, where propagation occurs in a dual space defined by the semiconvex transform \er{eq:op-D}, see for example \cite{M:08,DM1:15,ZD1:15,DMZ1:15,ZD2:15}.) 

In the specific case of the optimal control problem defined by \er{eq:value-W}, the properties \er{eq:op-G} and \er{eq:op-G-semigroup} may be used to directly propagate the value function $W_t$ to longer time horizons, with
\begin{align}
	& W_{t+\tau} = \op{G}_\tau^\oplus \, W_t\,,
	\qquad
	W_t = \op{G}_t^\oplus\, \psi
	\label{eq:op-G-DPP}
\end{align}
for any $t,\tau, t+\tau\in[0,t^*(P_0))$. In view of \er{eq:W} and \er{eq:op-G-DPP}, a particular solution $P_t$ of DRE \er{eq:DRE} satisfying the initial condition \er{eq:IC-M} can be similarly propagated forward in time. This gives rise to a characterization of $P_t$ in terms of the Hessian of the kernel $G_t$ of the max-plus primal-space fundamental solution $\op{G}_t^\oplus$. This characterization is referred to as a {\em max-plus primal space fundamental solution for DRE \er{eq:DRE}}. 

\if{false}

?? HERE ??

By definition \er{eq:kernel-G-def}, the kernel $G_t$ in Theorem \ref{thm:op-G} describes the value function $G_t(\cdot,y):\R^n\rightarrow\R^\pm$ of an optimal two-point boundary value problem, for any $t\in\R_{\ge 0}$, $y\in\R^n$ fixed. In this interpretation, $G_t(x,y)$ describes the maximal payoff \er{eq:payoff} achievable for dynamics \er{eq:dynamics}, subject to the boundary conditions $x(0) = x\in\R^n$ and $x(t) = y\in\R^n$. In general, this maximal payoff can be finite, or $\pm\infty$. 
However, Assumption \ref{ass:M} ensures that the value function of the auxiliary optimal control problem \er{eq:value-S} satisfies $S_t(x,\cdot)\in\cSave{M}$ for all $t\in\R_{\ge 0}$, $x\in\R^n$, by excluding the possibility of finite escape of the particular DRE solution \er{eq:Q-DRE}, \er{eq:Q-IC}. Consequently, the right-hand equality in \er{eq:kernel-G-def} provided by Theorem \ref{thm:op-G} implies, via the definition \er{eq:op-D} of the semiconvex transform, that $G_t(x,\cdot)\in\cSave{-M}$ for all $t\in\R_{\ge 0}$, $x\in\R^n$. Hence, by definition \er{eq:semiconvex} of $\cSave{-M}$,  payoff $G_t(x,y)$ can never take the value $+\infty$ for any $t\in\R_{\ge 0}$, $x,y\in\R^n$.

??

It is possible to further restrict the range of the kernel $G_t$ via a controllability assumption.  

\fi




\begin{theorem}
\label{thm:Lambda-and-Q}
Under Assumptions \ref{ass:M} and \ref{ass:controllable}, there exists a bijection $\Pi:\cSym^{2n\times 2n}\rightarrow\cSym^{2n\times 2n}$ such that  the kernel $G_t$ of \er{eq:kernel-G-def} takes the explicit finite quadratic form
\begin{align}
	& G_t(x,y)
	= \demi \left[ \ba{c}
		x \\ y
	\ea \right]' \Lambda_t \left[ \ba{c}
		x \\ y
	\ea \right]\in\R, \ \ 
	\Lambda_t \doteq \Pi^{-1}(Q_t)\,,
	\label{eq:kernel-G-and-Q}
\end{align}
for all $x,y\in\R^n$, $t\in\R_{>0}$, where $Q_t$ is as per \er{eq:Q-DRE}.
\end{theorem}
\begin{proof}
See Appendix \ref{app:Lambda-and-Q}.
\end{proof}

By inspection of Theorems \ref{thm:symplectic-and-Q} and \ref{thm:Lambda-and-Q}, the controllability Assumption \ref{ass:controllable} implies that the symplectic fundamental solution $\Sigma_t$ and the Hessian $\Lambda_t$ of the max-plus primal space fundamental solution kernel $G_t$ are equivalent. In particular, there exists a bijection $\Pi^{-1}\circ\Xi:\R^{2n\times 2n}\rightarrow\cSym^{2n\times 2n}$ such that
\begin{align}
	\Lambda_t 
	& = \Pi^{-1}\circ\Xi(\Sigma_t)\,,
	\qquad
	\Sigma_t = \Xi^{-1}\circ\Pi(\Lambda_t)
	\label{eq:Lambda-and-Sigma}
\end{align}
for all $t\in\R_{>0}$. Consequently, it is natural to expect that $\Lambda_t$ defines an alternative fundamental solution for DRE \er{eq:DRE}, \er{eq:IC-M}.

\begin{theorem}
\label{thm:P-and-Lambda}
Suppose that Assumptions \ref{ass:M} and \ref{ass:controllable} hold. Given any $P_0\in\cSym_{>M}^{n\times n}$, the corresponding unique solution $P_t$ of DRE \er{eq:DRE}, \er{eq:IC-M} exists and is given explicitly by
\begin{align}
	P_t
	& = 
			\Lambda_t^{11} - \Lambda_t^{12}\, (P_0 + \Lambda_t^{22})^{-1} (\Lambda_t^{12})',
	\label{eq:P-and-Lambda}
\end{align}
for all $t\in(0,t^*(P_0))$, where 
$\Lambda_t\in\cSym^{2n\times 2n}$
is as per \er{eq:kernel-G-and-Q}, 
and the maximal horizon of existence is $t^*(P_0)\in\R_{>0}^+$ is 
\begin{align}
	t^*(P_0)
	& = \sup\left\{ t\in\R_{>0} \left| \, P_0 + \Lambda_t^{22}\in\cSym_{<0}^{n\times n} \right. \right\}. 
	\label{eq:max-plus-escape}
\end{align}
\end{theorem}
\begin{proof}
See Appendix \ref{app:P-and-Lambda}.
\end{proof}

By inspection of \er{eq:symplectic-P}, \er{eq:symplectic-XY}, and \er{eq:kernel-G-and-Q}, \er{eq:P-and-Lambda}, it is evident that the symplectic and max-plus fundamental solutions both provide a characterization of all particular solutions of the DRE \er{eq:DRE}, \er{eq:IC-M}. Furthermore, both provide characterizations of the corresponding finite escape time $t^*(P_0)\in\R_{>0}^+$, see \er{eq:symplectic-escape} and \er{eq:max-plus-escape}. However, by inspection, a crucial difference between these latter characterizations concerns their ease of evaluation, assuming their respective fundamental solutions are known for all time. In particular, the existence or otherwise of a finite escape at time $t$ due to an initial condition $P_0\in\cSym_{>M}^{n\times n}$ can be verified using the max-plus characterization \er{eq:max-plus-escape} by testing the inequality  $P_0 + \Lambda_t^{22}\in\cSym_{<0}^{n\times n}$ once. However, the same verification using the symplectic characterization \er{eq:symplectic-escape} requires testing invertibility of $\Sigma_s^{11} + \Sigma_s^{12} P_0$ for all $s\in(0,t]$.


\newcommand{\ostar}		{{\circledast}}

\section{FUNDAMENTAL SOLUTIONS SEMIGROUPS}
\label{sec:fund-semigroups}
Both the symplectic fundamental solution $\Sigma_t\in\R^{2n\times 2n}$ and the max-plus fundamental solution $\Lambda_t\in\cSym^{2n\times 2n}$, specified respectively by \er{eq:symplectic-XY}, \er{eq:kernel-G-and-Q}, provide a path for establishing existence of a unique particular solution $P_t$ of DRE \er{eq:DRE}, \er{eq:IC-M} on the time interval $[0,t]\in\R_{>0}$, $t\in\R_{>0}$, and computing that solution. As the term {\em fundamental solution} implies, this is possible for any initial data $P_0$ satisfying \er{eq:IC-M}. Indeed, both fundamental solutions can be evolved to longer time horizons {\em independently} of any specific initial data for the DRE, thereby giving rise to a corresponding {\em symplectic} and {\em max-plus fundamental solution semigroups of matrices}. In defining the latter max-plus fundamental solution semigroup, it is useful to define a matrix operation $\ostar$ acting on $\Lambda,\hat\Lambda\in\cSym^{2n\times 2n}$ by 
\begin{equation}
	\begin{aligned}
	\lbrack\Lambda\, \ostar\, \hat\Lambda\rbrack^{11}
	& \doteq \Lambda^{11} - \Lambda^{12} (\hat\Lambda^{11} + \Lambda^{22})^{+} (\Lambda^{12})',
	\\
	[\Lambda\, \ostar\, \hat\Lambda]^{12}
	& \doteq -\Lambda^{12} (\hat\Lambda^{11} + \Lambda^{22})^{+} \hat\Lambda^{12},
	\\
	[\Lambda\, \ostar\, \hat\Lambda]^{21}
	& \doteq ([\Lambda\, \ostar\, \hat\Lambda]^{12})',
	\\
	[\Lambda\, \ostar\, \hat\Lambda]^{22}
	& \doteq \hat\Lambda^{22} - (\hat\Lambda^{12})' (\hat\Lambda^{11} + \Lambda^{22})^{+} \hat\Lambda^{12},
	\end{aligned}
	\label{eq:mult-ostar}
\end{equation}
using the notation of \er{eq:block}, in which $(\cdot)^+$ denotes the Moore-Penrose inverse.

\begin{theorem}
\label{thm:max-plus-semigroup}
Under Assumptions \ref{ass:M} and \ref{ass:controllable}, the families of matrices $\{\Sigma_t\}_{t\in\R_{>0}}$ and $\{\Lambda_t\}_{t\in\R_{>0}}$ defined by \er{eq:symplectic-XY} and \er{eq:kernel-G-and-Q}, and related via the bijection $\Pi^{-1}\circ\Xi$ of \er{eq:Lambda-and-Sigma}, define a pair of one-parameter semigroups of matrices in $\R^{2n\times 2n}$ satisfying
\begin{align}
	& \Sigma_{t+s} 
	= \Sigma_t\, \Sigma_s,
	&&
	\Lambda_{t+s}
	= \Lambda_t\, \ostar\, \Lambda_s,
	\label{eq:semigroup-associations}
\end{align}
for all $t,s\in\R_{>0}$, in which the respective associative binary operations are standard matrix multiplication, and the matrix operation $\ostar$ of \er{eq:mult-ostar}.
\end{theorem}
\begin{proof}
Fix $t,s\in\R_{>0}$. The left-hand semigroup property in \er{eq:semigroup-associations} is immediate by definition \er{eq:symplectic-XY} of the symplectic fundamental solution $\Sigma_t$. With Assumptions \ref{ass:M} and \ref{ass:controllable} asserted, Theorem \ref{thm:Lambda-and-Q} implies that $\Lambda_t,\Lambda_s,\Lambda_{t+s}\in\cSym^{2n\times 2n}$ are well-defined by \er{eq:kernel-G-and-Q}, while \er{eq:Lambda-and-Sigma} holds with bijection $\Pi^{-1}\circ\Xi$ by Theorems \ref{thm:symplectic-and-Q} and \ref{thm:Lambda-and-Q}. Furthermore, Theorem \ref{thm:op-G} and \er{eq:op-G-semigroup} imply that for any $\psi\in\dom(\op{G}_{t+s}^\oplus)\subset\dom(\op{G}_s^\oplus)$,
\begin{align}
	& \op{G}_t^\oplus\, \op{G}_s^\oplus \psi
	= \int_{\R^n}^\oplus G_t(\cdot,\eta) \otimes \left[ \int_{\R^n}^\oplus G_s(\eta,y)\otimes \psi(y) \, dy\right] d\eta 
	\nn\\
	& = \int_{\R^n}^\oplus \left[ \int_{\R^n}^\oplus G_t(\cdot,\eta) \otimes G_s(\eta,y) \, d\eta \right] \otimes \psi(y) \, dy
	\nn\\
	& = \op{G}_{t+s}^\oplus\, \psi = \int_{\R^n}^\oplus G_{t+s}(\cdot,y)\otimes \psi(y)\, dy.
	\nn
\end{align}
Applying an appropriate modification of \cite[Lemma 4.5]{M:08} to equate the kernels of the left- and right-hand sides above, Theorem \ref{thm:Lambda-and-Q} implies that
\begin{align}
	& \demi \left[ \ba{c}
		x \\ y
	\ea \right]' \Lambda_{t+s} \left[ \ba{c}
		x \\ y
	\ea \right]
	= G_{t+s}(x,y)
	\nn\\
	&
	 = \int_{\R^n}^\oplus G_t(x,\eta) \otimes G_s(\eta,y) \, d\eta
	\nn\\
	& = \int_{\R^n}^\oplus \demi \left[ \ba{c}
		x \\ \eta
	\ea \right]' \Lambda_t \left[ \ba{c}
		x \\ \eta
	\ea \right]  \otimes \demi \left[ \ba{c}
		\eta \\ y
	\ea \right]' \Lambda_s \left[ \ba{c}
		\eta \\ y
	\ea \right]
	d\eta
	\nn\\
	& = \int_{\R^n}^\oplus \lambda_{x,y}(\eta) \, d\eta
	\label{eq:Lambda-and-lambda}
\end{align}
where $\lambda_{x,y}:\R^n\rightarrow\R$ is defined for each $x,y\in\R^n$ by
\begin{align}
	& \lambda_{x,y}(\eta)
	\doteq \demi \left[ \ba{c}
			x \\ y \\\hline \eta
		\ea \right]' \! \left[ \ba{cc|c} 
			\Lambda_t^{11} & 0 & \Lambda_t^{12}
			\\
			&& \\[-3mm]
			0 & \Lambda_s^{22} & (\Lambda_s^{12})'
			\\[-3.5mm]
			&&
			\\\hline
			&& \\[-3.5mm]
			(\Lambda_t^{12})' & \Lambda_s^{12} & \Lambda_s^{11} + \Lambda_t^{22}
		\ea \right] \! \left[ \ba{c}
			x \\ y \\\hline \eta
		\ea \right]
	\nn\\
	& = \demi\, \eta' (\Lambda_s^{11} + \Lambda_t^{22}) \, \eta
		+ \eta' \left[ \ba{c} 
			\Lambda_t^{12} \\ (\Lambda_s^{12})'
	\ea \right]' \left[ \ba{c}
			x \\ y
	\ea \right]
	\nn\\
	& \qquad + 		
	 \demi \left[ \ba{c}
			x \\ y 
		\ea \right]' \left[ \ba{cc}
			\Lambda_t^{11} & 0 
			\\
			0 & \Lambda_s^{22}
		\ea \right] \left[ \ba{c}
			x \\ y 
		\ea \right]
	\nn
\end{align}
for all $\eta\in\R^n$. As $\Lambda_{t+s}\in\R^{2n\times 2n}$ is well-defined by \er{eq:kernel-G-and-Q}, note that $G_{t+s}(x,y)\in\R$ for any $x,y\in\R^n$ fixed, see also Lemma \ref{lem:kernel-G-finite}. That is, $\sup_{\eta\in\R^n} \lambda_{x,y}(\eta)\in\R$. Consequently, applying \cite[Lemma E.2]{DM1:15}, the following properties hold:
\begin{enumerate}[1)]
\item $\Lambda_s^{11} + \Lambda_t^{22}\in\cSym_{\le 0}^{n\times n}$;
\item the Moore-Penrose pseudo-inverse $(\Lambda_s^{11} + \Lambda_t^{22})^+\in\cSym_{\le 0}^{n\times n}$ exists; and
\item there exists a $\eta^*\in\R^n$ given by
\begin{align}
	& \eta^*\doteq - (\Lambda_t^{22} + \Lambda_s^{11})^+ \left[ \ba{c} 
			\Lambda_t^{12} \\ (\Lambda_s^{12})'
	\ea \right]' \left[ \ba{c}
			x \\ y
	\ea \right]
	\nn
\end{align}
such that
\end{enumerate}
\begin{align}
	\int_{\R^n}^\oplus \lambda_{x,y}(\eta) \, d\eta
	= \lambda_{x,y}(\eta^*)
	=  \demi \left[ \ba{c}
			x \\ y 
		\ea \right]' \! \left[ \ba{cc}
			\Lambda_t^{11} & 0 
			\\
			0 & \Lambda_s^{22}
		\ea \right] \! \left[ \ba{c}
			x \\ y 
		\ea \right] &
	\nn\\
	- \demi\, 
		\left[ \ba{c}
			x \\ y 
		\ea \right]' \left[ \ba{c}
			\Lambda_t^{12} \\ (\Lambda_s^{12})'
		\ea \right]
		(\Lambda_s^{11} + \Lambda_t^{22})^+
		\left[ \ba{c}
			\Lambda_t^{12} \\ (\Lambda_s^{12})'
		\ea \right]' \left[ \ba{c}
			x \\ y
		\ea \right] & .
	\nn
\end{align}
Applying this last property in \er{eq:Lambda-and-lambda} and recalling that $x,y\in\R^n$ are arbitrary yields \er{eq:mult-ostar} via
\begin{align}
	& \Lambda_{t+s} = \left[ \! \ba{cc}
			\Lambda_t^{11} & 0 
			\\
			0 & \Lambda_s^{22}
		\ea \! \right] -
	\left[ \! \ba{c}
			\Lambda_t^{12} \\ (\Lambda_s^{12})'
		\ea \! \right] \! 
		(\Lambda_s^{11} + \Lambda_t^{22})^+ \!
		\left[ \! \ba{c}
			\Lambda_t^{12} \\ (\Lambda_s^{12})'
		\ea \! \right]' \!\!\!.
	\nn
\end{align}
\end{proof}

The semigroups properties \er{eq:semigroup-associations} also naturally define respective notions of exponentiation. In particular,
\begin{equation}
	\begin{aligned}
	& \Sigma_t = (\Sigma_1)^{t},
	&& \Sigma_1 = \exp(\op{H}),
	\\
	& \Lambda_t = (\Lambda_1)^{\ostar t},
	&& \Lambda_1 = \Pi^{-1}\circ\Xi\circ\exp(\op{H}), 
	\end{aligned}
	\label{eq:exponentiation}
\end{equation}
in which $\op{H}$, $\Xi$, $\Pi^{-1}$ are as per \er{eq:symplectic-XY}, \er{eq:Xi}, \er{eq:inv-Pi}, and the exponentiations $(\cdot)^t$, $(\cdot)^{\ostar t}$ denote (respectively) the standard matrix exponentiation, and an exponentiation defined with respect to the $\ostar$ operation of \er{eq:mult-ostar}, see \cite[Section 5]{M:08} and Remark \ref{rem:exp} below. As is the case with standard matrix exponentiation, note that \er{eq:semigroup-associations}, \er{eq:exponentiation} imply that
\begin{align}
	& (\Lambda_1)^{\ostar (t+s)}
	= \Lambda_{t+s} = \Lambda_t\, \ostar\, \Lambda_s
	= (\Lambda_1)^{\ostar t} \, \ostar \, (\Lambda_1)^{\ostar s}\,.
	\nn
\end{align}
for all $t,s\in\R_{>0}$.

\begin{remark}\cite[Section 5]{M:08} 
\label{rem:exp}
The semigroup property \er{eq:semigroup-associations} immediately facilitates the definition of $\ostar$-exponentiation for any positive integer $n\in\N$ by
\begin{align}
	& (\Lambda_\tau)^{\ostar n}
	\doteq \underbrace{\Lambda_\tau\, \ostar\, \Lambda_\tau \, \ostar\, \cdots\, \ostar\, \Lambda_\tau}_{\text{$n$ times}} = \Lambda_{n \tau}
	\label{eq:ostar-integer}
\end{align}
where $\tau\in\R_{>0}$. By inspection, $((\Lambda_\tau)^{\ostar m})^{\ostar n} = (\Lambda_\tau)^{\ostar m n}$ for all $m,n\in\N$. Using this observation, \er{eq:ostar-integer} can be extended to positive rational and subsequently positive real exponents. In particular, given $p\in\cQ_{>0}$ and coprime $m,n\in\N$ such that $n\, p = m$, \er{eq:ostar-integer} implies that
$
	(\Lambda_\tau)^{\ostar m}
	= \Lambda_{m n (\tau/n)}
	= ((\Lambda_{\tau/n})^{\ostar m})^{\ostar n}
	\doteq ((\Lambda_\tau)^{\ostar p})^{\ostar n}
$,
That is, the positive rational $\ostar$-exponent $(\Lambda_\tau)^{\ostar p}$ is uniquely defined by 
\begin{align}
	& (\Lambda_\tau)^{\ostar p}\doteq (\Lambda_{\tau/n})^{\ostar m}
	\label{eq:ostar-rational}
\end{align}
for all $p = m/n\in\cQ_{>0}$, $m,n\in\N$ coprime. As $\cQ$ is dense in $\R$, and the map $\tau\mapsto\Lambda_\tau$, $\tau\in\R_{>0}$, is continuous by \er{eq:kernel-G-and-Q}, it immediately follows that $\Lambda_t = \lim_{p\in\cQ_{>0}, \, p\rightarrow  t} \Lambda_{p}$. As $\Lambda_p$ can be replaced with the $\ostar$-expononent $(\Lambda_1)^{\ostar p}$ of \er{eq:ostar-rational}, the $\ostar$-exponent $(\Lambda_1)^{\ostar t}$ of \er{eq:exponentiation} is uniquely defined by
\begin{align}
	& (\Lambda_1)^{\ostar t}
	\doteq \lim_{p\in\cQ_{>0}, \, p\rightarrow  t} (\Lambda_1)^{\ostar p}
	= \lim_{p\in\cQ_{>0}, \, p\rightarrow  t} (\Lambda_{1/n})^{\ostar m}
	\label{eq:ostar-real}
\end{align}
for all $t\in\R_{>0}$, identically to \cite{M:08}. Note that in the right-hand equality of \er{eq:ostar-real}, coprime $m,n\in\N$ are uniquely defined for each $p\in\cQ_{>0}$ in the limiting sequence. Where $t\in\R_{>0}$ is irrational, it follows immediately that $m,n\rightarrow\infty$.
\end{remark}


\newcommand{\oer}[1]			{\normalfont\eqno{\er{#1}}}
\newcommand{\ooer}[1]			{\normalfont\tag{\ref{#1}}}

\section{SOLVING THE DRE \er{eq:DRE}, \er{eq:IC-M}} 
\label{sec:example}

Theorems \ref{thm:P-and-Lambda} and \ref{thm:max-plus-semigroup} together describe a new max-plus primal space fundamental solution semigroup of matrices $\{\Lambda_t\}_{t\in\R_{>0}}$ for propagating solutions $P_t\in\cSym^{n\times n}$ of DRE \er{eq:DRE} forward in time $t\in\R_{>0}$ from initializations $P_0\in\cSym_{>M}^{n\times n}$ as per \er{eq:IC-M}. In particular, \er{eq:mult-ostar}, \er{eq:semigroup-associations}  describe propagation of the new fundamental solution for this DRE, while \er{eq:P-and-Lambda} specifies how this fundamental solution may be used to evaluate a particular solution at any time $t\in\R_{>0}$. In addition, \er{eq:max-plus-escape} provides a general characterization of the corresponding maximal horizon of existence $t^*(P_0)\in\R_{>0}^+$. By inspection, this characterization allows easy verification of whether a specific time $t$ falls before a finite escape $t^*(P_0)\in\R_{>0}$ (if it exists), by testing if $P_0 + \Lambda_t^{22}\in\cSym_{<0}^{n\times n}$ at that time. This is simpler than the corresponding verification using the characterization \er{eq:symplectic-escape} provided by the symplectic fundamental solution \er{eq:symplectic-XY}, where invertibility of a matrix over a range of times must be tested. 

\subsection{Recipe}
\label{ssec:recipe}
A recipe that uses the one-parameter max-plus primal space fundamental semigroup $\{\Lambda_t\}_{t\in\R_{>0}}$ to compute the solution $P_{t}$ of DRE \er{eq:DRE} for any initialization 
$$
	P_0\in\cSym_{>M}^{n\times n}
	\oer{eq:IC-M}
$$
of the form \er{eq:IC-M}, using a fixed time step $\delta\in\R_{>0}$, is as follows:

\noindent\rule{\columnwidth}{0.5pt}
\parbox[t]{0.05\columnwidth}{{\sf I.}\hfill}%
\parbox[t]{0.9\columnwidth}{\hfil{\centering\sf Initialize and propagate the semigroup \er{eq:mult-ostar}, \er{eq:semigroup-associations}}\hfil}
\\[-2mm]
\rule{\columnwidth}{0.5pt}
\begin{itemize}
\item[\done] {\em (Initialize basis)} Select $M\in\cSym^{n\times n}$ of \er{eq:basis} using Lemma \ref{lem:ass-M-holds}. Check that Assumptions \ref{ass:M} and \ref{ass:controllable} hold.
\item[\dtwo] {\em (Initialize semigroup)} Fix time step $\delta\in\R_{>0}$ and maximal time horizon $\bar t\doteq K\delta\in\R_{>0}$ for some fixed $K\in\N$. Using the matrix operators $\Xi$, $\Pi^{-1}$ of \er{eq:Xi}, \er{eq:inv-Pi}, initialize an element of the semigroup by
\begin{align}
	& \Lambda_\delta = \Pi^{-1}\circ\Xi\circ\exp(\op{H}\delta)\in\cSym^{2n\times 2n},
	\label{eq:Lambda-delta}
\end{align}
where $\op{H}\in\R^{2n\times 2n}$ is the Hamiltonian matrix \er{eq:symplectic-XY}.
\item[\dthree] {\em (Propagate semigroup)} Compute a subset $\{\Lambda_{k\delta}\}_{k\in\N_{\le K}}$ of the semigroup, corresponding to a temporal grid defined by $\delta\in\R_{>0}$, via the evolution 
\begin{align}
	& \Lambda_{(k+1)\delta}
	= \Lambda_\delta\, \ostar \, \Lambda_{k\delta},
	\qquad k\in\N_{< K}, 
	\label{eq:Lambda-iteration}
\end{align}
as per \er{eq:semigroup-associations}. (See also Remark \ref{rem:doubling} below.)
\end{itemize}
\rule{\columnwidth}{0.5pt}
\parbox[t]{0.05\columnwidth}{{\sf II.}\hfill}%
\parbox[t]{0.9\columnwidth}{\hfil{\centering\sf Solve the DRE \er{eq:DRE}, \er{eq:IC-M}}\hfil}
\\[-2mm]
\rule{\columnwidth}{0.5pt}
\begin{itemize}
\item[\dfour] {\em (Initialize a solution)} Select $P_0\in\cSym_{>M}^{n\times n}$. Set $k=1$.
\item[\dfive] {\em (Test for finite escape)} If $P_0 + \Lambda_{k\delta}^{22}\in\cSym_{<0}^{n\times n}$ as per \er{eq:max-plus-escape} then  
evaluate the solution $P_{k\delta}$ at time step $k$ as
\begin{align}
	& P_{k\delta} 
	= \Lambda_{k\delta}^{11} - \Lambda_{k\delta}^{12} \, (P_0 + \Lambda_{k\delta}^{22})^{-1} (\Lambda_{k\delta}^{12})'
	\label{eq:P-delta}
\end{align}
as per \er{eq:P-and-Lambda}. Otherwise, record a finite escape time as occurring in the interval $((k-1)\delta,k\delta]$ and exit.
\item[\dsix] {\em (Iterate)} Increment $k$. If $k\in\N_{\le K}$ then go to step {\dfive}. Otherwise, exit.
\end{itemize}
${ }^{ }$\\[-6mm]
\rule{\columnwidth}{0.5pt}

As indicated, the recipe consists of a total of 6 steps, divided into two parts. Part {\sf I} concerns the initialization and propagation of a subset of the one parameter semigroup of matrices $\{\Lambda_t\}_{t\in\R_{>0}}$ required for computing {\em any} particular solution of DRE \er{eq:DRE}, \er{eq:IC-M} up to a pre-specified time horizon $\bar t\in\R_{>0}$. Part {\sf II} concerns the subsequent evaluation of such a particular solution (and corresponding finite escape, if it exists). Crucially, part {\sf I} need only be completed once, with the elements of the semigroup computed there used repeatedly in part {\sf II} in evaluating any particular solutions of interest, without modification. This is demonstrated by example.

\begin{remark}
\label{rem:doubling}
For fast propagation of $\Lambda_{k\delta}$ to large $k\in\N_{\le K}$, where $K\doteq2^N$, $N\in\N$, the {\em linear} time-index accumulation in \er{eq:Lambda-iteration} can be replaced with time-index {\em doubling} \cite{ZD2:15}, ie.
$$
	\Lambda_{(2^{\kappa+1})\delta} = \Lambda_{(2^\kappa)\delta} \otimes \Lambda_{(2^\kappa)\delta}.
	\quad \kappa\in\{0\}\cup\N_{<N}. 
$$
\end{remark}

\subsection{Example -- no finite escape}
In demonstrating the recipe described above, an example from \cite{M:08} is considered. In particular, define $A\in\R^{n\times n}$, $B\in\R^{n\times m}$, $C\in\R^{p\times n}$, and $n=m=p=2$, by
\begin{align}
	& 
	\begin{gathered}
	A \doteq \left[ \ba{cc}
			-2.000 & +1.600
			\\
			-1.600 & -0.400 
		\ea \right],\ 
	B \doteq \left[ \ba{cc}
			+0.216 & -0.008
			\\
			-0.008 & +0.216
		\ea \right]^\frac{1}{2}\!\!\!,
	\\
	C \doteq \left[ \ba{cc}
			+1.500 & +0.200
			\\
			+0.200 & +1.600
		\ea \right]^{\frac{1}{2}}\!\!\!,
	\ 
	M_0 \doteq \left[ \ba{cc}
		+0.651 & -0.310
		\\
		-0.310 & +1.160
		\ea \right],
	\end{gathered}
	\nn
\end{align}
where $M_0\in\cSym_{>0}^{n\times n}$ is the corresponding stabilizing solution of ARE \er{eq:ARE} as per Lemma \ref{lem:ass-M-holds}. In view of \er{eq:ass-M-holds}, select
\begin{align}
	& M \doteq \left[ \ba{cc}
		-1.000 & -0.200
		\\
		-0.200 & -1.000
		\ea \right] \in\cSym_{<M_0}^{n\times n}.
	\nn
\end{align}
Consequently, Assumption \ref{ass:M} holds. Step {\done} is completed via a standard rank calculation to verify that Assumption \ref{ass:controllable} holds. Theorems \ref{thm:symplectic-and-Q}, \ref{thm:Lambda-and-Q}, and \ref{thm:max-plus-semigroup} subsequently imply that the one parameter semigroup of matrices $\{\Lambda_t\}_{t\in\R_{>0}}$ propagated by \er{eq:mult-ostar}, \er{eq:semigroup-associations} is well-defined and may be computed as indicated in steps {\dtwo}, {\dthree}. With $\delta \doteq 0.05$, $K\doteq 80$, $\bar t \doteq 4$, \er{eq:Lambda-delta} yields
\begin{align}
	\Lambda_\delta
	& = 
	\left[ \ba{cc|cc}
		-83.48 & -3.021 & +92.26 & -4.011
		\\
		-3.021 & -91.11 & +11.07 & +92.42
		\\\hline
		+92.26 & +11.07 & -102.6 & -3.420
		\\
		-4.011 & +92.42 & -3.420 & -94.28
	\ea \right].
	\nn
	%
\end{align}
Subsequently iterating via \er{eq:Lambda-iteration} as per step {\dthree} yields the required set of matrices $\{\Lambda_{k\delta}\}_{k\in\N_{\le K}}$. 
\if{false}

These matrices are partially illustrated in Figure \ref{fig:L22} via the corresponding quadratic functions $x\mapsto \demi x' \Lambda_{k\delta}^{22} x$, $x\in\R^2$, for selected $k\in\N_{\le K}$. Observe in particular that $\Lambda_{k\delta}^{22}$ is non-decreasing in $k$.

\begin{figure}[h]
\begin{center}
\vspace{-4mm}
\epsfig{file=DRE-Lambda.eps,width=85mm}
\vspace{-3mm}
\caption{Quadratic functions $x\mapsto \demi x' \Lambda_{k\delta}^{22} x$, $x\in\R^2$, $k\in\N_{\le K}$.}
\label{fig:L22}
\end{center}
\vspace{-5mm}
\end{figure}

\fi

In order to demonstrate the computation of particular solutions of DRE \er{eq:DRE} in steps {\dfour} -- {\dsix}, select an initialization
\begin{align}
	P_0
	& \doteq -0.1 \, I
	\in\cSym_{>M}^{n\times n}
	\label{eq:P0-no-escape}
\end{align}
as per step {\dfour} and \er{eq:IC-M}. Iterating through $k\in\N_{\le K}$ as per steps {\dfive} and {\dsix}, testing for finite escape and applying \er{eq:P-delta}, yields the computed solution $P_{k\delta}$, $k\in\N_{\le K}$ of DRE \er{eq:DRE}, \er{eq:P0-no-escape}. This solution, along with corresponding symplectic and {\sf RK45} solutions, is illustrated in Figure \ref{fig:no-escape}. (Here, the {\sf MATLAB${}^\text{\tiny TM}$ RK45} solver is used, with absolute and relative tolerances set to $10^{-12}$.) All three solutions are in reasonable agreement. No finite escape is observed.


\subsection{Example -- finite escape}
In order to illustrate finite escape phenomenon, the set of matrices $\{\Lambda_{k\delta}\}_{k\in\N_{\le K}}$ computed above is reused to evaluate the particular solution of DRE \er{eq:DRE}, \er{eq:IC-M} for the initial condition
\begin{align}
	P_0
	& \doteq \left[ \ba{cc}
		2.000 & 0.000
		\\
		0.000 & 6.500
	\ea \right]
	\in\cSym_{>M}^{n\times n}.
	\label{eq:P0-escape}
\end{align}
The problem data is otherwise unchanged. Using the initialization \er{eq:P0-escape} in step {\dfour} and iterating steps {\dfive} and {\dsix} yields the corresponding DRE solution. A finite escape is demonstrated to occur within the horizon $\bar t = 4$ of computation, with $t^*(P_0) \in (2.8,2.9]$ established using \er{eq:max-plus-escape}. Figure \ref{fig:max-eig-escape} illustrates $\sigma_{\max} (P_0 + \Lambda_t^{22})$, $t=k\delta$, $k\in\N_{\le K}$, where $\sigma_{\max}:\cSym^{n\times n}\rightarrow\R$ denotes the maximum eigenvalue map. Note specifically that zero crossing occurs at the finite escape time, as per \er{eq:max-plus-escape}. Note further that $t\mapsto\sigma_{\max}(P_0 + \Lambda_t^{22})$ defines a monotone non-decreasing function. This monotonicity follows from that used to establish the representation \er{eq:max-plus-escape} of the finite escape time $t^*(P_0)$, see the proof of Theorem \ref{thm:P-and-Lambda}. It guarantees that no finite escape occurs prior to this zero crossing.

The computed solution $P_{t}$ of DRE \er{eq:DRE}, \er{eq:P0-escape} for $t=k\delta$, $k\in\N_{\le K}$, is illustrated in Figure \ref{fig:escape}, along with the corresponding symplectic and {\sf RK45} solutions. These solutions are in good agreement, as measured by the absolute errors illustrated in Figure \ref{fig:errors}. As may be observed in the latter figure, these errors increase immediately prior to the finite escape time as the entries of $P_t\in\cSym^{n\times n}$ diverge to $\pm\infty$. For brevity, an error analysis is not included.

\begin{figure}[ht!]
\begin{center}
\psfrag{PP11}{$P_t^{11}$}
\psfrag{PP12}{$P_t^{12}$}
\psfrag{PP22}{$P_t^{22}$}
\epsfig{file=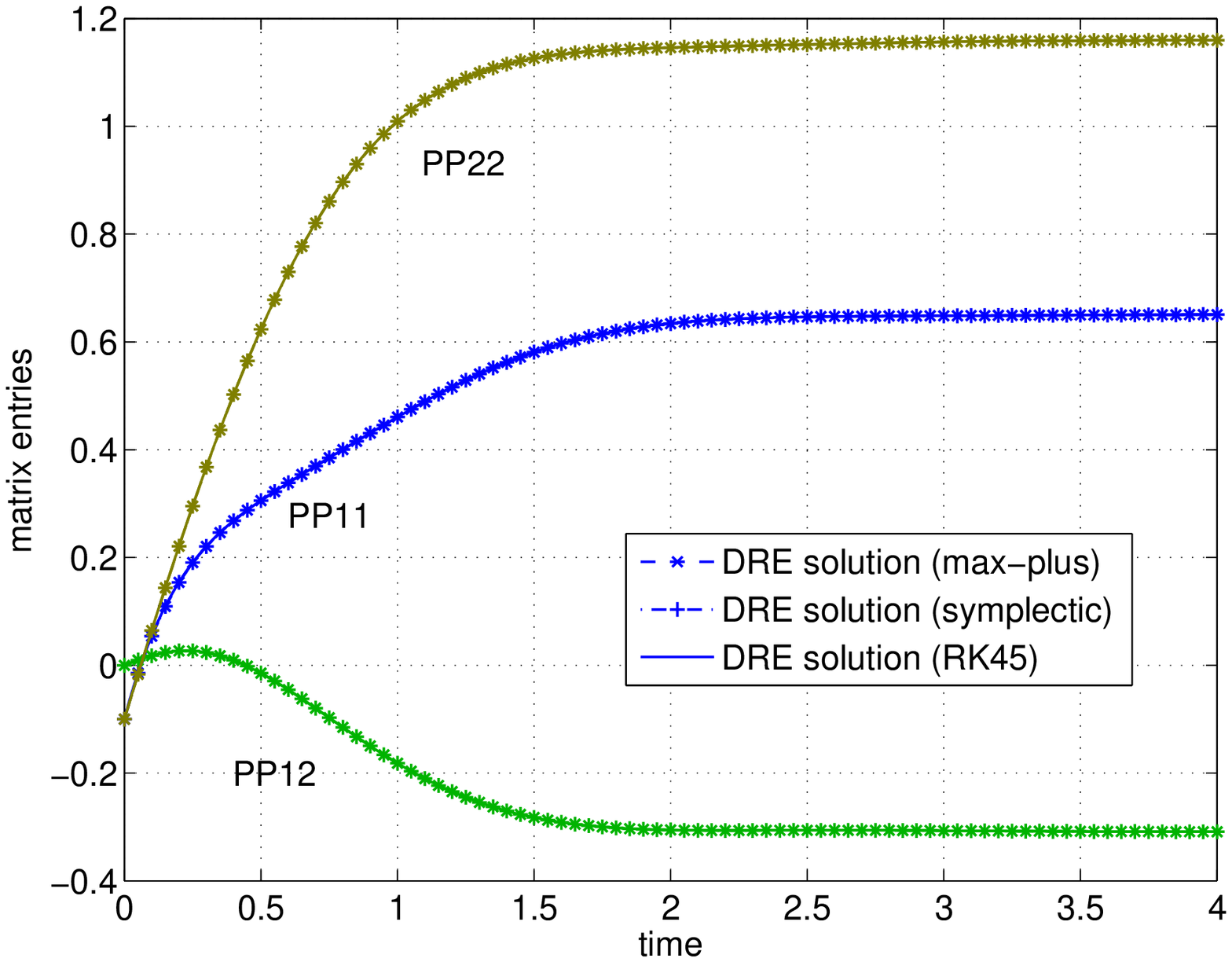,width=88mm}
\vspace{-7mm}
\caption{Max-plus and {\sf RK45} solutions of DRE \er{eq:DRE}, \er{eq:P0-no-escape}.}
\label{fig:no-escape}
\end{center}
\vspace{-5mm}
\end{figure}

\begin{figure}[h]
\begin{center}
\psfrag{SSSSSSSSSSSSSSSSSSSS}{\small $t\mapsto\sigma_{\max}(P_0 + \Lambda_t^{22})$}
\epsfig{file=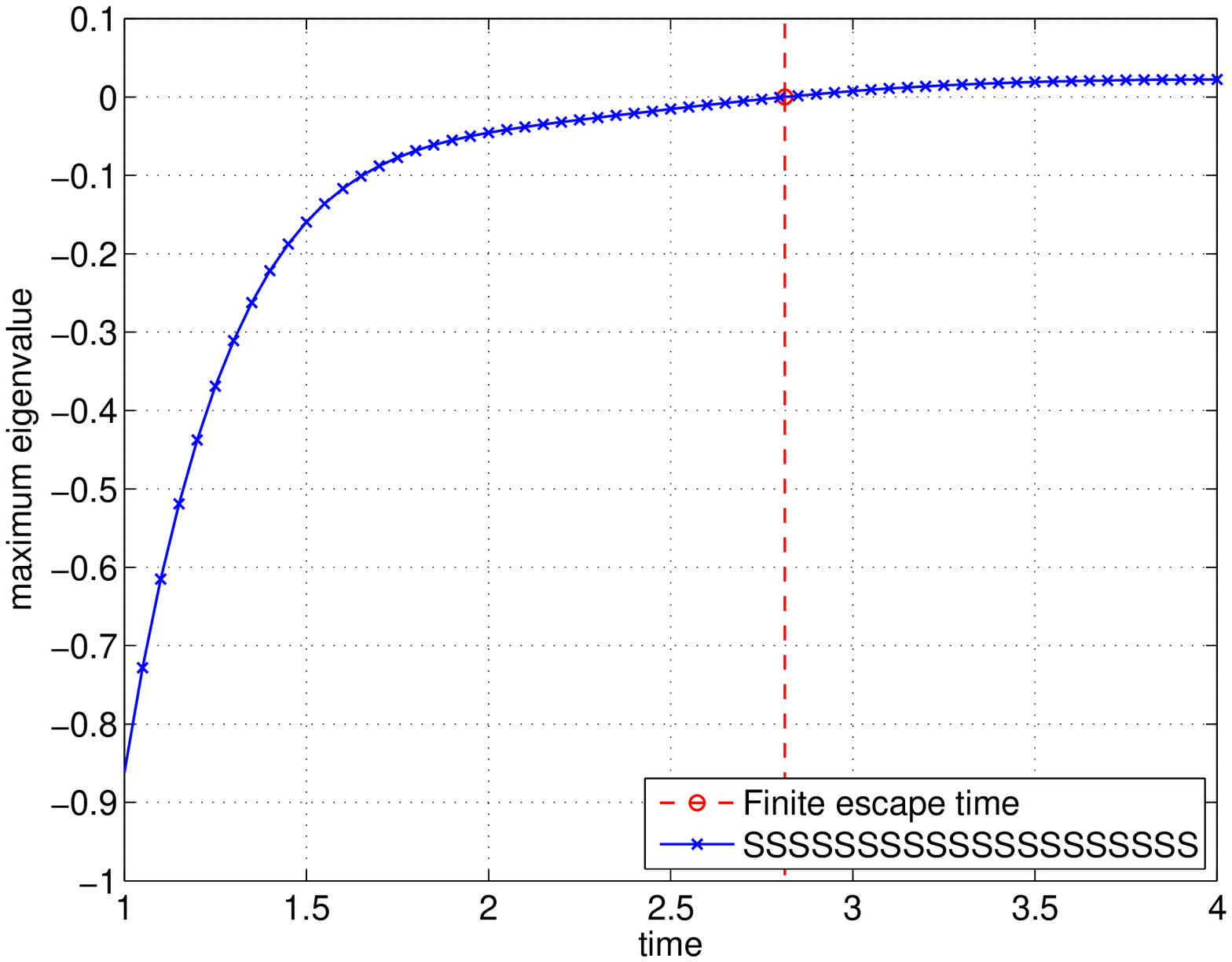,width=88mm}
\vspace{-7mm}
\caption{Maximum eigenvalue map $t\mapsto\sigma_{\max}(P_0 + \Lambda_t^{22})$, $t=k\delta$.}
\label{fig:max-eig-escape}
\end{center}
\vspace{-5mm}
\end{figure}

\begin{figure}[h]
\begin{center}
\psfrag{PP11}{$P_t^{11}$}
\psfrag{PP12}{$P_t^{12}$}
\psfrag{PP22}{$P_t^{22}$}
\epsfig{file=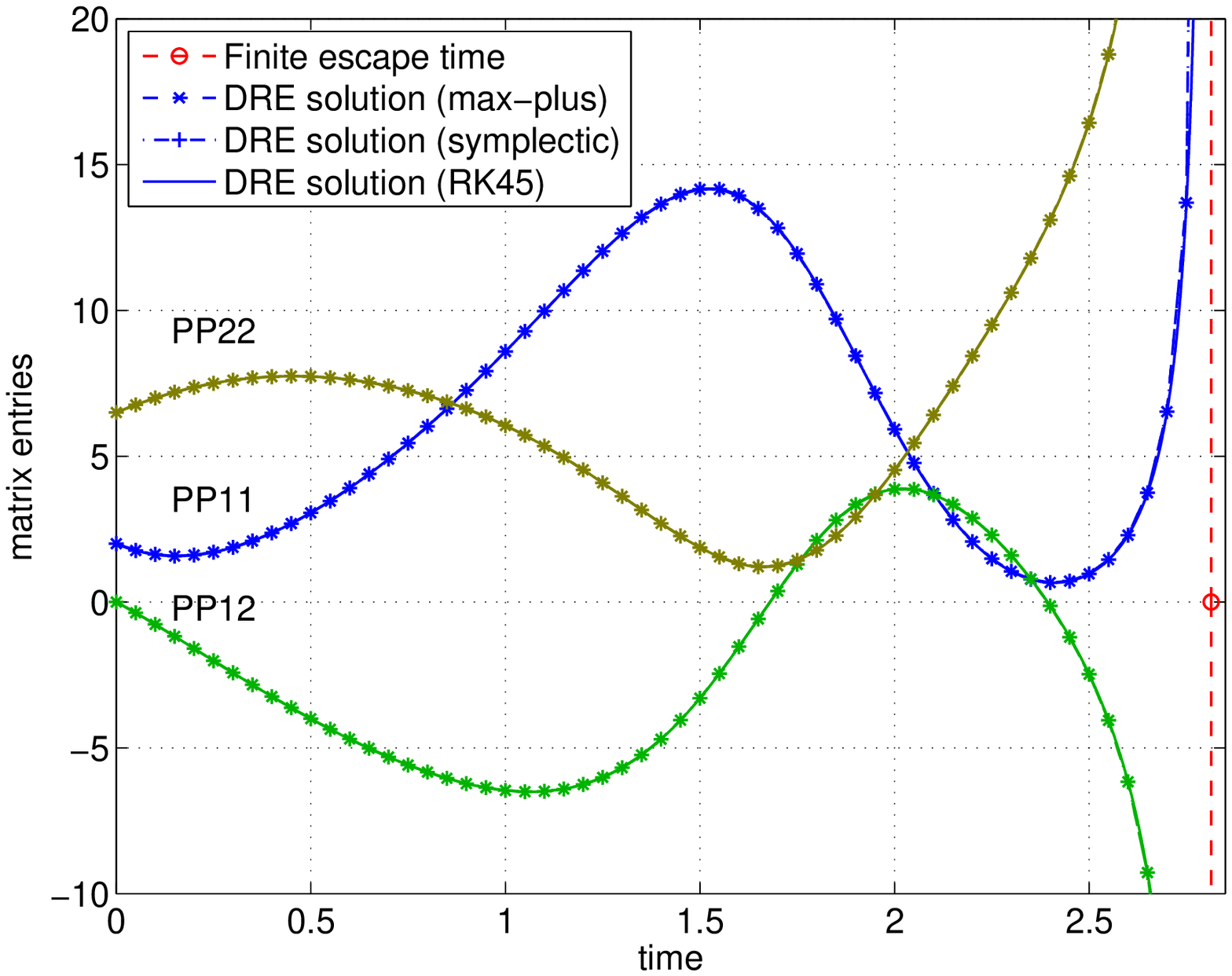,width=88mm}
\vspace{-8mm}
\caption{Max-plus and {\sf RK45} solutions of DRE \er{eq:DRE}, \er{eq:P0-escape}.}
\label{fig:escape}
\end{center}
\vspace{-5.5mm}
\end{figure}

\begin{figure}[h]
\begin{center}
\epsfig{file=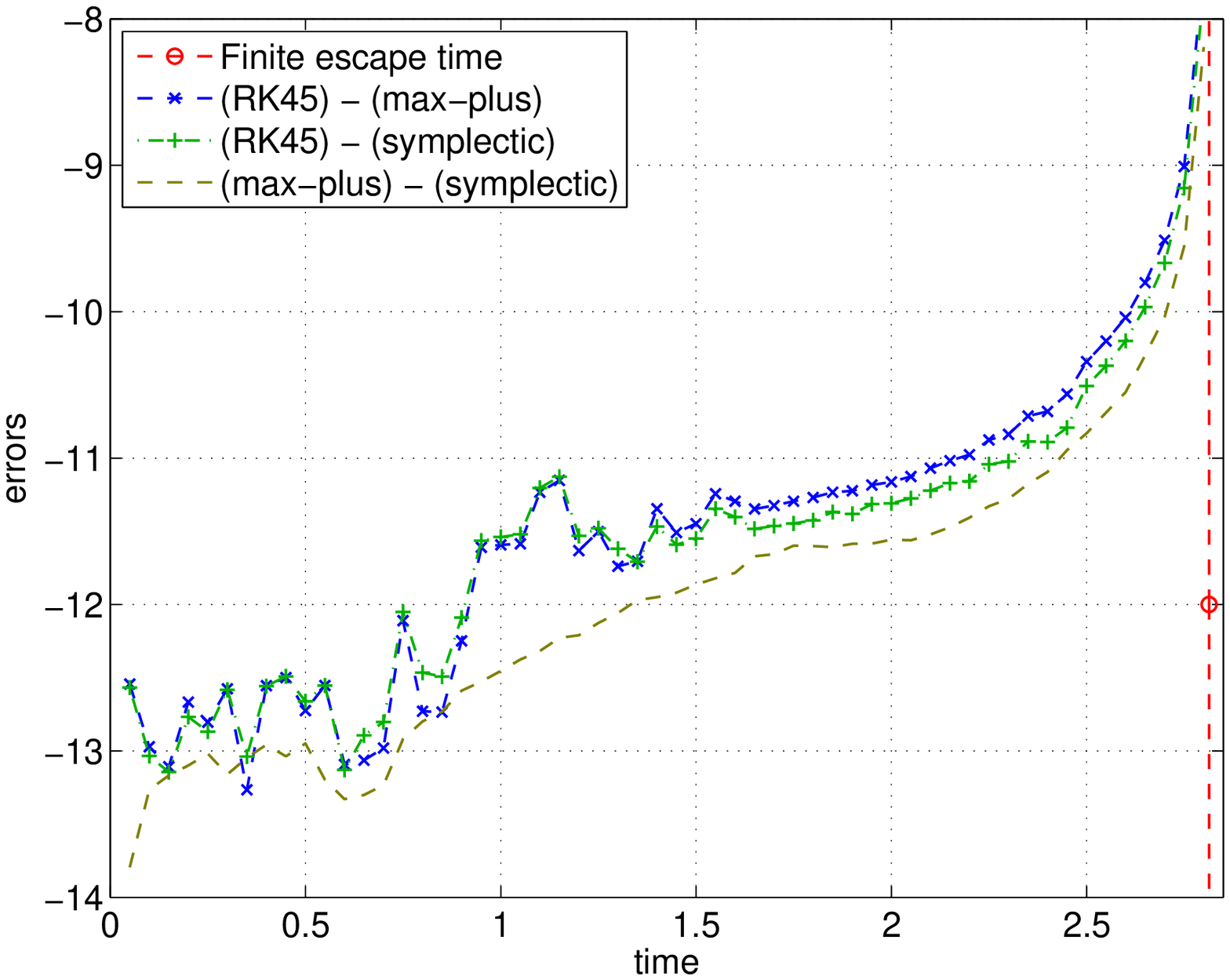,width=88mm}
\vspace{-7mm}
\caption{$\log_{10}(\|\cdot\|_2)$ of errors between solutions of DRE \er{eq:DRE}, \er{eq:P0-escape}.}
\label{fig:errors}
\end{center}
\vspace{-5mm}
\end{figure}


\section{CONCLUSIONS}
\label{sec:conc}

A new fundamental solution for a class of differential Riccati equations (DREs) is developed using tools from max-plus and semiconvex analysis. It is shown that this fundamental solution is defined by a corresponding fundamental solution semigroup, which describes the evolution of all particular solutions of the DRE, on all time horizons. A new characterization of finite escape time is also provided, enabling a simpler test for existence of particular solutions in comparison with the standard symplectic fundamental solution.



\appendix


\subsection{Proof of Theorem \ref{thm:symplectic-and-Q}}
\label{app:symplectic-and-Q}

Since $Q_t\in\R^{2n\times 2n}$ satisfies DRE \er{eq:Q-DRE}, \er{eq:Q-IC} for all $t\in\R_{\ge 0}$, see \er{eq:equal-escape}, it may be represented by a corresponding symplectic fundamental solution of the form \er{eq:symplectic-XY}, denoted here by $\widehat\Sigma_t\in\R^{4n\times 4n}$. In order to apply \er{eq:symplectic-XY}, define $\ophat{H},\Delta\in\R^{4n\times 4n}$ by
\begin{align}
	\ophat{H}
	& \doteq \left[ \ba{c|c}
		 -\hat A  &	-\hat B \hat B' 
		 \\[-4mm]
		 &
		 \\\hline
		 & \\[-3mm]
		 \ \hat C' \hat C	&	\hat A'
	\ea \right],
	\quad
	\Delta
	\doteq \left[ \ba{cc|cc}
		I & 0 & 0 & 0 \\
		0 & 0 & I & 0 \\\hline
		0 & I & 0 & 0 \\
		0 & 0 & 0 & I
	\ea \right],
	\nn
\end{align}
where $\hat A\in\R^{2n\times 2n}$, $\hat B\in\R^{2n\times m}$, $\hat C\in\R^{p\times 2n}$ are as per \er{eq:ABC-hat}.
Note by inspection that $\Delta = \Delta' = \Delta^{-1}$. By substitution, a straightforward calculation yields that
\begin{align}
	\ophat{H}
	& = \Delta \left[ \ba{c|c}
			\op{H} & 0 \\\hline
			0 & 0 \ea \right] \Delta\,,
	\nn
\end{align}
where $\op{H}\in\R^{2n\times 2n}$ is as per \er{eq:symplectic-XY}. Hence, the symplectic fundamental solution $\widehat\Sigma_t$ for DRE \er{eq:Q-DRE} is, again by \er{eq:symplectic-XY},
\begin{align}
	& \widehat\Sigma_t
	= \exp(\ophat{H} t)
	= \Delta \left[ \ba{c|c}
			\exp(\op{H} t) & 0 \\\hline
			0 & I
	\ea \right] \Delta
	\nn\\
	& 
	= \Delta \left[ \ba{c|c}
			\Sigma_t & 0 \\\hline
			0 & I
	\ea \right] \Delta
	= \left[ \ba{cc|cc}
		\Sigma_t^{11} & 0 & \Sigma_t^{12} & 0
		\\
		0 & I & 0 & 0 
		\\\hline
		\Sigma_t^{21} & 0 & \Sigma_t^{22} & 0 
		\\
		0 & 0 & 0 & I
	\ea \right],
	\label{eq:symplectic-hat}
\end{align}
for all $t\in\R_{\ge 0}$, where the notation of \er{eq:block} has been applied. Hence, the particular solution $Q_t$ of DRE \er{eq:Q-DRE}, \er{eq:Q-IC} is given in terms of the symplectic fundamental solution \er{eq:symplectic-XY}, with respect to $\widehat\Sigma_t$, by
\begin{align}
	& Q_t
	= \widehat Y_t \widehat X_t^{-1}
	\label{eq:pre-Q-t-1}
\end{align}
for all $t\in[0,t^*(Q_0))\equiv\R_{\ge 0}$, see \er{eq:equal-escape}, in which 
\begin{align}
	& \left[ \ba{c}
		\widehat X_t
		\\[-4mm]
		\\\hline 
		\\[-3mm]
		\widehat Y_t
	\ea \right] 
	\doteq \widehat\Sigma_t \left[ \ba{c}
					I \\\hline \\[-3mm] \mu(M)
				\ea \right]
	\nn\\
	& = \left[ \ba{cc|cc}
		\Sigma_t^{11} & 0 & \Sigma_t^{12} & 0
		\\
		0 & I & 0 & 0 
		\\\hline
		\Sigma_t^{21} & 0 & \Sigma_t^{22} & 0 
		\\
		0 & 0 & 0 & I
	\ea \right] \left[ \ba{cc}
		I & 0 \\
		0 & I \\\hline
		+M & -M \\
		-M & +M
	\ea \right]
	\nn\\
	& = \left[ \ba{cc}
			\Sigma_t^{11} + \Sigma_t^{12} M & -\Sigma_t^{12} M
			\\
			0 & I
			\\\hline
			\Sigma_t^{21} + \Sigma_t^{22} M & -\Sigma_t^{22} M
			\\
			-M & +M
	\ea \right] \in \R^{4n \times 2n},
	\label{eq:pre-Q-t-2}
\end{align}
and $\mu(M)$ is defined by \er{eq:mu}. For any fixed $t\in\R_{\ge 0}$, note in particular that
\begin{align}
	\widehat X_t^{-1}
	& = \left[ \ba{cc}
			(\Sigma_t^{11} + \Sigma_t^{12} M)^{-1} & (\Sigma_t^{11} + \Sigma_t^{12} M)^{-1} \Sigma_t^{12} M
			\\
			0 & I
	\ea \right],
	\nn
\end{align}
in which $(\Sigma_t^{11} + \Sigma_t^{12} M)^{-1}$ is well-defined as $t^*(M)=+\infty$ by Assumption \ref{ass:M} and \er{eq:equal-escape}. That is, $\widehat X_t^{-1}$ is well-defined. 
Its substitution in \er{eq:pre-Q-t-1}, along with $\widehat Y_t$ from \er{eq:pre-Q-t-2}, yields $Q_t = \widehat Y_t \widehat X_t^{-1} \doteq \Xi(\Sigma_t)$, where $\Xi:\R^{2n\times 2n}\rightarrow\R^{2n\times 2n}$ is defined by
\begin{align}
	& \Xi(\Sigma)
	\doteq\left[\ba{cc}
		\Xi^{11}(\Sigma) & \Xi^{12}(\Sigma)
		\\
		\Xi^{21}(\Sigma) & \Xi^{22}(\Sigma)
	\ea\right],
	\label{eq:Xi}
	\\
	& \dom(\Xi)
	\doteq
	\left\{\Sigma\in\R^{2n\times 2n}\, \left| \ba{c}
							\Sigma^{11}+\Sigma^{12}M\in\R^{n\times n}\\
							\text{invertible}
	\ea \right. \right\},
	\nn
\end{align}
using the notation of \er{eq:block}, with 
\begin{align}
	\Xi^{11}(\Sigma)
	& \doteq (\Sigma^{21}+\Sigma^{22}M)(\Sigma^{11}+\Sigma^{12}M)^{-1}\,,
	\nn\\
	\Xi^{12}(\Sigma)
	& \doteq 
	\Xi^{11}(\Sigma) \,
	\Sigma^{12} M - \Sigma^{22} M\,,
	\nn\\
	\Xi^{21}(\Sigma)
	& \doteq -M(\Sigma^{11}+\Sigma^{12}M)^{-1}\,,
	\nn\\
	\Xi^{22}(\Sigma)
	& \doteq 
	\Xi^{21}(\Sigma)\,
	\Sigma^{12}M+M.
	\nn
\end{align}
As $M$ is invertible by Assumption \ref{ass:M}, it may be verified directly that $\Xi$ of \er{eq:Xi} is invertible, with $\Xi^{-1}:\R^{2n\times 2n}\rightarrow\R^{2n\times 2n}$ given by
\begin{align}
	& \Xi^{-1}(Q)
	\doteq\left[\ba{cc}
		(\Xi^{-1})^{11}(Q) & (\Xi^{-1})^{12}(Q)
		\\
		(\Xi^{-1})^{21}(Q) & (\Xi^{-1})^{22}(Q)
	\ea\right],
	\label{eq:inv-Xi}
	\\
	& \dom(\Xi^{-1})
	\doteq \left\{ Q\in\R^{2n\times 2n} \, \biggl| \, Q^{21}\in\R^{n\times n} \text{ invertible} \right\}.
	\nn
\end{align}
where
\begin{align}
	(\Xi^{-1})^{11}(Q)
	& \doteq -(Q^{21})^{-1}Q^{22}
	\nn\\
	(\Xi^{-1})^{12}(Q)
	& \doteq -(Q^{21})^{-1}(M-Q^{22})M^{-1}
	\nn\\
	(\Xi^{-1})^{21}(Q)
	& \doteq 
	V^{11}\,
	(\Xi^{-1})^{11}(Q)
	+Q^{12} 
	\nn\\
	(\Xi^{-1})^{22}(Q)
	& \doteq 
	Q^{11}\,
	(\Xi^{-1})^{12}(Q)
	-Q^{12}M^{-1}\,.
	\nn
\end{align}
That is, \er{eq:symplectic-and-Q} holds. 
${ }^{ }$\hfill{\small $\blacksquare$}

\subsection{Proof of Lemma \ref{lem:ass-M-holds}}
\label{app:ass-M-holds}


Fix $M_0\in\cSym_{\ge 0}^{n\times n}$ as the stabilizing solution of ARE \er{eq:ARE} indicated in the lemma statement.
Let $t^*(M_0)\in\R_{>0}^+$ denote the maximal horizon of existence \er{eq:symplectic-escape} of the DRE 
\begin{align}
	\dot R_t
	& = A' R_t + R_t A + R_t B B' R_t + C'C, \quad R_0 = M_0.
	\label{eq:R-DRE}
\end{align}
As $M_0$ is the stabilizing solution of ARE \er{eq:ARE}, note that $R_t \doteq M_0$ is the unique solution of this DRE for all $t\in\R_{\ge 0}$. That is, $t^*(M_0) = +\infty$. Choose any invertible $M\in\cSym^{n\times n}$ such that \er{eq:ass-M-holds} holds, and note that such a choice is always possible. Recalling \er{eq:Q-11}, let $Q_t^{11}\in\cSym^{n\times n}$, $t\in[0,t^*(M))$ denote the unique solution of DRE \er{eq:Q-11} initialized with $Q_0^{11} = M$. As DREs \er{eq:Q-11} and \er{eq:R-DRE} are identical, Lemma \ref{lem:monotone} and \er{eq:ass-M-holds} imply that solutions $Q_t^{11}$ and $R_t$ satisfy the monotonicity property
\begin{align}
	& Q_t^{11} - R_t = Q_t^{11} - M_0 \in\cSym_{<0}^{n\times n}
	\label{eq:high-side}
\end{align}
for all $t\in[0,t^*(M))$. By inspection, this provides an upper bound for $Q_t^{11}$. In order to determine a lower bound, choose $w_s = 0$ for all $s\in[0,t]$ suboptimal in the definition \er{eq:value-S} of $S_t(x,0)$. Recalling \er{eq:op-S}, \er{eq:payoff}, \er{eq:S},
\begin{align}
	& \demi x' Q_t^{11} x = S_t(x,0)
	\ge 
	\demi x' O_t \, x,
	\label{eq:low-side}
\end{align}
in which $O_t\in\cSym^{n\times n}$ is well-defined by 
\begin{align}
	O_t 
	& \doteq \! \int_0^t \! \exp(A' s) C' C \exp(A\, s) \, ds + \exp(A' t) M \exp(A \, t)
	\nn
\end{align}
for all $t\in\R_{\ge 0}$. Note that $O_t\in\cSym^{n\times n}$ is finite for all $t\in\R_{\ge 0}$, and provides a lower bound for $Q_t^{11}\in\cSym^{n\times n}$. Hence, combining \er{eq:high-side} and \er{eq:low-side},
\begin{align}
	Q_t^{11}\in \cSym_{\ge O_t}^{n\times n} \cap \cSym_{<M_0}^{n\times n}
	\nn
\end{align}
for all $t\in[0,t^*(M))$. A simple contradiction argument subsequently implies that $Q_t^{11}\in\cSym^{n\times n}$ is finite for all $t\in\R_{\ge 0}$, so that $t^*(M) = +\infty$.
${ }^{ }$\hfill{\small $\blacksquare$}


\subsection{Proof of Lemma \ref{lem:kernel-G-finite}}
\label{app:kernel-G-finite}

Suppose that Assumption \ref{ass:M} holds. Fix $x,y\in\R^n$, $t\in\R_{>0}$. Note that $G_t(x,y)\in\R^-$ by Theorem \ref{thm:op-G}.

{\em (Necessity)} Suppose that $G_t(x,y)\in\R$. Recalling the value function interpretation of $G_t(x,y)$, if the dynamics \er{eq:dynamics} are not controllable from $x$ to $y$ in time $t$, it immediately follows by definition \er{eq:kernel-G-def} that $G_t(x,y) = -\infty$. Hence, the dynamics \er{eq:dynamics} must be controllable from $x$ to $y$ in time $t$. Necessity follows as $x,y\in\R^n$ and $t\in\R_{>0}$ are arbitrary.

{\em (Sufficiency)} Suppose that dynamics \er{eq:dynamics} are controllable. Consequently, Lemma \ref{lem:controllable-Q22} implies that  $Q_t^{22} \in\cSym_{>M}^{n\times n}$, where $Q_t^{22}$ is as per \er{eq:Q-22}. Consequently, $S_t(x,\cdot)\in\cSvex{-M} = \dom(\op{D}_\varphi)$, so that $\op{D}_\varphi S_t(x,\cdot)\in\cSave{-M}$ is well defined. So, applying the semiconvex transform \er{eq:op-D} to $S_t(x,\cdot)$ yields
\begin{align}
	& (\op{D}_\varphi S_t(x,\cdot))(y)
	= -\! \int_{\R^n}^\oplus \varphi(\xi,y)\otimes (-S_t(x,\xi)) \, d\xi
	\nn\\
	& = -\! \int_{\R^n}^\oplus \demi \left[ \ba{c}
							\xi \\ y 
						\ea \right]' \mu(M) \left[ \ba{c}
							\xi \\ y 
						\ea \right] 
			- \demi \left[ \ba{c}
					x \\ \xi
				\ea \right]' Q_t \left[ \ba{c}
					x \\ \xi
				\ea \right]
			d\xi
	\nn\\
	& = -\! \int_{\R^n}^\oplus \!\! \demi \! \left[ \ba{c}
						x \\ y \\\hline \xi
					\ea \right]' \!\! \left[ \ba{cc|c}
						-Q_t^{11} & 0 & -Q_t^{12} 
						\\
						0 & +M & - M
						\\\hline
						&& \\[-4mm]
						-(Q_t^{12})' & -M & M - Q_t^{22}
					\ea \right] \!\! \left[ \ba{c}
						x \\ y \\\hline \xi
					\ea \right] \! d\xi
	\nn\\
	& = - \demi \left[ \ba{c}
				x \\ y
			\ea \right]' \left[ \ba{cc}
				-Q_t^{11} & 0
				\\
				0 & +M
			\ea \right] \left[ \ba{c} 
				x \\ y
			\ea \right]
	\nn\\
	& \quad
		+ \demi \left[ \ba{c}
				x \\ y
			\ea \right]' \left[ \ba{c}
				-Q_t^{12}
				\\
				-M
			\ea \right] (M - Q_t^{22})^{-1}
			\left[ \ba{c}
				-Q_t^{12}
				\\
				-M
			\ea \right]'
			\left[ \ba{c}
				x \\ y
			\ea \right]
	\nn\\
	& \doteq
	\demi \left[ \ba{c}
				x \\ y
			\ea \right]' \Lambda_t
		\left[ \ba{c}
				x \\ y
			\ea \right],
	\label{eq:pre-G-and-Lambda}
\end{align}
where $(M-Q_t^{22})^{-1}$ is guaranteed to exist by Lemma \ref{lem:controllable-Q22}, so that $\Lambda_t\in\R^{2n\times 2n}$ by definition. Hence, applying the right-hand equality of \er{eq:kernel-G-def} of Theorem \ref{thm:op-G},
\begin{align}
	G_t(x,y)
	& = \demi \left[ \ba{c}
				x \\ y
			\ea \right]' \Lambda_t
		\left[ \ba{c}
				x \\ y
			\ea \right] \in\R,
	\label{eq:G-and-Lambda}
\end{align}
thereby completing the proof.
${ }^{ }$\hfill{\small $\blacksquare$}


\begin{lemma}
\label{lem:controllable-Q22}
Under Assumption \ref{ass:M}, controllability of the dynamics \er{eq:dynamics} implies that $Q_t^{22} \in\cSym_{>M}^{n\times n}$ for all $t\in\R_{>0}$. 
\end{lemma}
\begin{proof}{\em (Lemma \ref{lem:controllable-Q22})}
With $M\in\cSym^{n\times n}$ satisfying Assumption \ref{ass:M}, recall that $t^*(M) = +\infty$ as per \er{eq:equal-escape}. Consequently, the optimal dynamics associated with $S_t(x,y)$ of \er{eq:value-S}, \er{eq:S} are well-defined by the time-dependent ODE
\begin{align}
	\dot x_s^*
	& = (A + B B' Q_{t-s}^{11} ) \, x_s^*\,,
	\qquad x_0 = x\,,
	\label{eq:x-star}
\end{align}
for all $s\in[0,t]$. Let $\op{V}_t:\Delta_{0,t}\rightarrow\R^{n\times n}$ denote the evolution operator associated with \er{eq:x-star}, with $\Delta_{0,t} \doteq \{ (r,s)\in\R_{\ge 0}^2 \, \bigl| \, 0\le r \le s \le t \}$. By definition, see for example \cite[Proposition 3.6, p.138]{BDDM:07},
\begin{equation}
	\begin{aligned}
	\op{V}_t(\sigma,\sigma) & = I\,,
	\\
	\ts{\pdtone{}{s}} \op{V}_t(s,\sigma) & = (A + B B' Q_{t-s}^{11} ) \, \op{V}_t(s,\sigma)\,,
	\\
	\ts{\pdtone{}{\sigma}} \op{V}_t(s,\sigma) & = -\op{V}_t(s,\sigma) \, (A + B B' Q_{t-\sigma}^{11} )\,,
	\end{aligned}
	\label{eq:op-V-properties}
\end{equation}
for all $(s,\sigma)\in\Delta_{0,t}$. Define $\op{U}_t:\Delta_{0,t}\rightarrow\R^{n\times n}$ via \er{eq:op-V-properties} by
\begin{align}
	\op{U}_t(r,\tau)
	& \doteq \op{V}_t(t-\tau,t-r)'
	\label{eq:op-U}
\end{align}
for all $(r,\tau)\in\Delta_{0,t}$. By inspection of \er{eq:op-V-properties}, \er{eq:op-U},
\begin{equation}
	\begin{aligned}
	\op{U}_t(\tau,\tau) & = I\,,
	\\
	\ts{\pdtone{}{r}} \op{U}_t(r,\tau) & = [\ts{\pdtone{}{\sigma}} \op{V}_t(s,\sigma) \bigl|_{(s,\sigma)=(t-\tau, t-r)}]' \, (-1) 
	\\
	& = (A + B B' Q_{r}^{11} )' \, \op{V}_t(t-\tau,t-r)'
	\\
	& = (A + B B' Q_{r}^{11} )' \, \op{U}_t(r,\tau)\,,
	\\
	\ts{\pdtone{}{\tau}} \op{U}_t(r,\tau) & = [\ts{\pdtone{}{s}} \op{V}_t(s,\sigma) \bigl|_{(s,\sigma)=(t-\tau, t-r)}]' \, (-1) 
	\\
	& = - \op{V}_t(t-\tau,t-r)' (A + B B' Q_{\tau}^{11})' 
	\\
	& = - \op{U}_t(r,\tau) \, (A + B B' Q_{\tau}^{11})' 
	\end{aligned}
\end{equation}
That is, $\op{U}_t:\Delta_{0,t}\rightarrow\R^{n\times n}$ is the evolution operator for the dynamics associated with $(A + B B' Q_{s}^{11})'$, $s\in[0,t]$. Comparing with \er{eq:Q-12}, it immediately follows that $Q_s^{12} = - \op{U}_t(s,0) M$ for all $s\in[0,t]$. Hence, \er{eq:Q-22} implies that
\begin{align}
	& Q_t^{22} - M
	= \int_0^t (Q_s^{12})' B B' Q_s^{12} \, ds
	\nn\\
	& \quad = \int_0^t M \, \op{U}_t(s,0)' B B' \, \op{U}_t(s,0) M \, ds
	= M \, \op{C}_t \, M
	\label{eq:Q-22-minus-M}
\end{align}
where $\op{C}_t \doteq \int_0^t \op{V}_t(t,t-s) B B' \, \op{V}_t(t,t-s)\, ds\in\cSym_{\ge 0}^{n\times n}$ is the controllability gramian for the pair $(A+B B' Q_{t-\cdot}^{11}, B)$ on $[0,t]$, by definition of $\op{V}_t$. However, recall that controllability is preserved under state feedback, see for example \cite[p.48]{AM:71}. Hence, $(A,B)$ completely controllable implies that $(A+B B' Q_{t-\cdot}^{11}, B)$ is completely controllable, which in turn implies that $\op{C}_t$ is invertible for $t\in\R_{>0}$. That is, $\op{C}_t\in\cSym_{>0}^{n\times n}$ for all $t\in\R_{>0}$. As $M$ is invertible by Assumption \ref{ass:M}, the assertion immediately follows by \er{eq:Q-22-minus-M}.
\end{proof}


\subsection{Proof of Theorem \ref{thm:Lambda-and-Q}}
\label{app:Lambda-and-Q}

Fix any $t\in\R_{>0}$, $x\in\R^n$. Applying Lemma \ref{lem:kernel-G-finite}, and in particular \er{eq:pre-G-and-Lambda}, \er{eq:G-and-Lambda}, it follows immediately that $Q_t\in\R^{2n\times 2n}$, $\Lambda_t\in\cSym^{2n\times 2n}$ of \er{eq:Q-DRE}, \er{eq:pre-G-and-Lambda} are related via
\begin{align}
	Q_t & = \Pi(\Lambda_t)\,,
	\qquad
	\Lambda_t = \Pi^{-1}(Q_t)\,,
	\nn
\end{align}
with matrix operators $\Pi, \Pi^{-1}:\cSym^{2n\times 2n}\rightarrow\cSym^{2n\times 2n}$ defined using the notation of \er{eq:block} by
\begin{align}
	& \Pi(\Lambda)
	\doteq
	\nn\\
	& \! \left[ \!\! \ba{cc} 
		\Lambda^{11} - \Lambda^{12}(M+ \Lambda^{22})^{-1} (\Lambda^{12})' & \!\! \Lambda^{12} (M+ \Lambda^{22})^{-1} M
		\\
		M(M+ \Lambda^{22})^{-1} (\Lambda^{12})' & \!\! M - M ( M + \Lambda^{22})^{-1} M
	\ea \!\! \right]
	\nn\\
	& \dom(\Pi)
	\doteq  \left\{\Lambda\in\cSym^{2n\times 2n}\, \biggl| \, \Lambda^{22}\in\cSym_{<-M}^{n\times n}\right\},
	\label{eq:Pi}
	\\
	& \Pi^{-1}(Q)
	\doteq 
	\nn\\
	& \! \left[ \!\! \ba{cc}
		Q^{11} \! + Q^{12}(M - Q^{22})^{-1} (Q^{12})' & \!\!\! Q^{12} ( M - Q^{22})^{-1}M
		\\
		M (M - Q^{22})^{-1} (Q^{12})' & \!\!\! M ( M - Q^{22})^{-1} M \! - M
	\ea \!\! \right]
	\nn\\
	& \dom(\Pi^{-1})
	\doteq \left\{Q\in\cSym^{2n\times 2n}\, \biggl| \, Q^{22}\in\cSym_{>M}^{n\times n}\right\}.
	\label{eq:inv-Pi}
\end{align}
It may be verified directly that $\Pi\circ\Pi^{-1}$ is the identity.


\subsection{Proof of Theorem \ref{thm:P-and-Lambda}}
\label{app:P-and-Lambda}

Throughout, it is assumed that Assumptions \ref{ass:M} and \ref{ass:controllable} hold, with $M\in\cSym^{n\times n}$ specified by the former, as per the theorem statement. Note in particular that $t^*(M) = +\infty$, so that $(\Sigma_t^{11} + \Sigma_t^{12}\, M)^{-1}$ exists for all $t\in\R_{\ge 0}$, where $\Sigma_t$ is the symplectic fundamental solution identified in \er{eq:symplectic-XY}. Consequently, $Q_t\in\cSym^{2n\times 2n}$ is well-defined as the unique solution of DRE \er{eq:Q-DRE}, \er{eq:Q-IC}, for all $t\in\R_{\ge 0}$ by Assumption \ref{ass:M}, see Theorem \ref{thm:symplectic-and-Q} and its proof. Note that $P_0\in\cSym_{>M}^{n\times n} = \dom(\Upsilon)$ by hypothesis and \er{eq:Upsilon}.

The proof proceeds by demonstrating a sequence of implications concerning the following claims, posed with respect to arbitrary fixed $t\in\R_{>0}$ and $P_0\in\cSym_{>M}^{n\times n}$:
\begin{enumerate}[\em 1)]\itemsep=1mm
\item $t\in (0,t^*(P_0))$;
\item $\Upsilon(P_0) + Q_s^{22}\in\cSym_{<0}^{n\times n}$ for all $s\in(0,t]$;
\item $\Upsilon(P_0) + Q_t^{22}\in\cSym_{<0}^{n\times n}$;
\item $P_0 + \Lambda_t^{22}\in\cSym_{<0}^{n\times n}$; and
\item \er{eq:P-and-Lambda} and \er{eq:max-plus-escape} hold.
\end{enumerate}
In particular, it is shown that {\em 1) $\Leftrightarrow$ 2) $\Leftrightarrow$ 3) $\Leftrightarrow$ 4) $\Rightarrow$ 5)}.

{\em 2) $\Rightarrow$ 1):} Suppose that $\Upsilon(P_0) + Q_s^{22} \in\cSym_{<0}^{n\times n}$ for all $s\in(0,t]$. Applying \er{eq:Upsilon} and Theorem \ref{thm:symplectic-and-Q}, 
\begin{align}
	& M^{-1} ( \Upsilon(P_0) + Q_s^{22} ) \, M^{-1}
	\nn\\
	& = (M - P_0)^{-1} - (\Sigma_s^{11} + \Sigma_s^{12} M)^{-1} \Sigma_s^{12}
	\label{eq:1=>2-eq-3}
\end{align}
where it may be noted that the inverses on the right-hand side are guaranteed to exist. By hypothesis, the left-hand side is invertible, so that a matrix $K_s\in\R^{n\times n}$ is well-defined for an arbitrary $s\in(0,t]$ by
\begin{align}
	K_s
	& \doteq 
	(\Sigma_s^{11} + \Sigma_s^{22} M)^{-1} + 
	(\Sigma_s^{11} + \Sigma_s^{22} M)^{-1} \Sigma_t^{12}
	\nn\\
	& \qquad \times \left[
		 (M - P_0)^{-1} - (\Sigma_s^{11} + \Sigma_s^{12} M)^{-1} \Sigma_s^{12}
	\right]^{-1}
	\nn\\
	& \qquad \times (\Sigma_s^{11} + \Sigma_s^{22} M)^{-1}\,.
	\nn
\end{align}
However, the Woodbury Lemma implies that
\begin{align}
	K_s
	& = \left[ (\Sigma_s^{11} + \Sigma_s^{22} M)  - \Sigma_s^{12} (M - P_0) \right]^{-1}
	\nn\\
	& = (\Sigma_s^{12} + \Sigma_s^{12} P_0)^{-1}\,.
	\nn
\end{align}
That is, $\Sigma_s^{12} + \Sigma_s^{12} P_0\in\cSym^{n\times n}$ is invertible. Recalling \er{eq:symplectic-escape}, and that $s\in(0,t]$ is arbitrary, immediately implies that {\em 1)} holds. 


{\em 1) $\Rightarrow$ 2):} 
Fix an arbitrary $t\in(0,t^*(P_0))$. Analogously to the proof of Theorem \ref{thm:symplectic-and-Q}, let $\widetilde Q_s\in\cSym^{2n\times 2n}$ denote the unique solution of DRE \er{eq:Q-DRE} subject to the initialization 
\begin{align}
	& \widetilde Q_0 = \mu(P_0)
	\label{eq:Q-tilde-IC}
\end{align}
defined, via \er{eq:mu}, for all $s\in [0, t^*(\widetilde Q_0))$, where $t^*(\widetilde Q_0)\in\R_{>0}$ is the corresponding maximal horizon of existence \er{eq:symplectic-escape}. Analogously to the argument yielding \er{eq:equal-escape}, observe that $t^*(\widetilde Q_0) = t^*(P_0)$, so that $t\in(0,t^*(\widetilde Q_0))$. An application of the symplectic fundamental solution \er{eq:symplectic-P}, \er{eq:symplectic-XY}, \er{eq:symplectic-hat}, yields 
\begin{align}
	& \widetilde Q_s = \widetilde Y_s \widetilde X_s^{-1}
	\label{eq:Q-tilde}
\end{align}
for all $s\in[0,t]$, in which
\begin{align}
	& \left[ \ba{c}
		\widetilde X_s
		\\[-4mm]
		\\\hline 
		\\[-3mm]
		\widetilde Y_s
	\ea \right] 
	\doteq \widehat\Sigma_s \left[ \ba{c}
					I \\\hline \\[-3mm] \mu(P_0)
				\ea \right]
	\nn\\
	& 
	= \left[ \ba{cc|cc}
		\Sigma_s^{11} & 0 & \Sigma_s^{12} & 0
		\\
		0 & I & 0 & 0 
		\\\hline
		& \\[-3mm]
		\Sigma_s^{21} & 0 & \Sigma_s^{22} & 0 
		\\
		0 & 0 & 0 & I
	\ea \right] \left[ \ba{cc}
		I & 0 \\
		0 & I \\\hline
		& \\[-3mm]
		+P_0 & -M \\
		-M & +M
	\ea \right]
	\nn\\
	& = \left[ \ba{cc}
			\Sigma_s^{11} + \Sigma_s^{12} P_0 & -\Sigma_s^{12} M
			\\
			0 & I
			\\\hline
			& \\[-3mm]
			\Sigma_t^{21} + \Sigma_s^{22} P_0 & -\Sigma_s^{22} M
			\\
			-M & +M
	\ea \right] \in \R^{4n \times 2n}.
	\nn
\end{align}
for all $s\in[0,t]$. In particular,
\begin{align}
	\widetilde X_s^{-1}
	& = \left[ \ba{cc}
			(\Sigma_s^{11} + \Sigma_s^{12} P_0)^{-1} & (\Sigma_s^{11} + \Sigma_s^{12} P_0)^{-1} \Sigma_s^{12} M
			\\
			0 & I
	\ea \right],
	\nn
\end{align}
in which $(\Sigma_s^{11} + \Sigma_s^{12} P_0)^{-1}$ is well-defined for all $s\in[0,t]$, as $t\in(0,t^*(P_0))$, see \er{eq:symplectic-escape}. Consequently, recalling \er{eq:block}, \er{eq:Q-tilde},
\begin{align}
	\widetilde Q_s^{22}
	& = M - M (\Sigma_s^{11} + \Sigma_s^{12} P_0)^{-1} \Sigma_s^{12} M
	\label{eq:1=>2-eq-1}
\end{align}
is well-defined for all $s\in[0,t]$. Recalling \er{eq:Q-IC} and \er{eq:Q-tilde-IC}, as $\widetilde Q_0 = \mu(P_0) \ge \mu(M) = Q_0$, monotonicity of DRE solutions (see for example Lemma \ref{lem:monotone}) implies that $\widetilde Q_s - Q_s\in\cSym_{\ge 0}^{2n\times 2n}$, so that in particular 
\begin{align}
	& \widetilde Q_s^{22} - Q_s^{22}\in\cSym_{\ge 0}^{n\times n}
	\label{eq:monotone-tilde}
\end{align}
for all $s\in[0,t]$. Fix an arbitrary $s\in(0,t]$. Rearranging \er{eq:1=>2-eq-1} and applying \er{eq:monotone-tilde}, Theorem \ref{thm:Lambda-and-Q}, and Lemma \ref{lem:controllable-Q22},
\begin{align}
	& (\Sigma_s^{11} + \Sigma_s^{12} P_0)^{-1} \Sigma_s^{12}
	= M^{-1} ( M - \widetilde Q_s^{22} ) \, M^{-1}
	\nn\\
	& \le M^{-1} (M - Q_s^{22} )\, M^{-1}\in\cSym_{<0}^{n\times n}\,.
	\label{eq:1=>2-eq-2}
\end{align}
Theorem \ref{thm:symplectic-and-Q} and \er{eq:Xi} implies via the notation of \er{eq:block} that
\begin{align}
	Q_s^{22}
	& =  [\Xi(\Sigma_s)]^{22} = M - M(\Sigma_s^{11} + \Sigma_s^{12} M)^{-1} \Sigma_s^{12} M.
	\label{eq:Q22-and-Xi22}
\end{align}
Recall that $\Sigma_s\in\dom(\Xi)$ (ie. the inverse involved is guaranteed to exist) by Assumption \ref{ass:M}, as $s\in(0,t^*(M)) \equiv \R_{>0}$. 
Furthermore, as $s\in(0,t^*(P_0))$, definition \er{eq:symplectic-escape} implies that $\Sigma_s^{11} + \Sigma_s^{12} P_0$ is invertible. Hence, a matrix $L_s\in\cSym^{n\times n}$ is well-defined by
\begin{align}
	& L_s
	\doteq (M - P_0) + (M-P_0) (\Sigma_s^{11} + \Sigma_s^{12} P_0)^{-1} \Sigma_s^{12} (M-P_0)
	\nn\\
	& = (M - P_0) + (M-P_0) 
	\nn\\
	& \qquad \times 
	\left[ (\Sigma_s^{11} + \Sigma_s^{12} M) - \Sigma_s^{12} (M - P_0) \right]^{-1}
	\Sigma_s^{12} \, (M-P_0).
	\nn
\end{align}
where the second equality follows by adding and subtracting $\Sigma_s^{12} M$ within the inverse. Applying \er{eq:1=>2-eq-2}, and the fact that $P_0\in\cSym_{>M}^{n\times n}$, note that $L_s\in\cSym_{<0}^{n\times n}$ by definition. The Woodbury Lemma subsequently implies that
\begin{align}
	L_s
	& = \left[ (M-P_0)^{-1} - (\Sigma_s^{11} + \Sigma_s^{12} M)^{-1} \Sigma_s^{12} \right]^{-1}
	\nn\\
	& = M ( \Upsilon(P_0) + Q_s^{22} )^{-1} M
	\nn
\end{align}
where the second equality follows as per \er{eq:1=>2-eq-3}. Consequently, as $M\in\cSym^{n\times n}$ is invertible and $L_s\in\cSym_{<0}^{n\times n}$,
\begin{align}
	& 
	\Upsilon(P_0) + Q_s^{22}
	= M L_s^{-1} \, M \in\cSym_{<0}^{n\times n} .
	\nn
\end{align}
As $s\in(0,t]$ is arbitrary, claim {\em 2)} immediately follows.


{\em 2) $\Rightarrow$ 3):} By hypothesis, $\Upsilon(P_0) + Q_s^{22}\in\cSym_{<0}^{n\times n}$ for all $s\in(0,t]$. Selecting $s=t$ yields claim {\em 3)} as required.


{\em 3) $\Rightarrow$ 2):} By hypothesis, $\Upsilon(P_0) + Q_t^{22}\in\cSym_{<0}^{n\times n}$. Furthermore, $\Upsilon(P_0)\in\cSym_{<-M}^{n\times n}$ by \er{eq:Upsilon}. Hence, $Q_t^{22}\in\cSym^{n\times n}$, so that $(Q_\sigma^{12})' B B' Q_\sigma^{12}$ must be integrable with respect to $\sigma\in[0,t]$ by definition \er{eq:Q-22}. In particular,
\begin{align}
	Q_t^{22} - M
	& = \int_0^t (Q_\sigma^{12})' B B' Q_\sigma^{12} \, d\sigma
	\nn\\
	& \ge \int_0^s (Q_\sigma^{12})' B B' Q_\sigma^{12} \, d\sigma
	= Q_s^{22} - M
	\nn
\end{align}
for any fixed $s\in(0,t]$.
Hence, $Q_s^{22} - Q_t^{22}\in\cSym_{\le 0}^{n\times n}$, so that
\begin{align}
	\Upsilon(P_0) + Q_s^{22}
	& = (\Upsilon(P_0) + Q_t^{22}) + (Q_s^{22}  - Q_t^{22})\in\cSym_{<0}^{n\times n}.
	\nn
\end{align}
Recalling that $s\in(0,t]$ is arbitrary yields claim {\em 2)} as required.

{\em 3) $\Rightarrow$ 4):} Recalling \er{eq:Upsilon} and Theorem \ref{thm:Lambda-and-Q}, see \er{eq:kernel-G-and-Q}, \er{eq:Pi},
\begin{align}
	& \Upsilon(P_0) + Q_t^{22}
	= (-M - M(P_0-M)^{-1} M) 
	\nn\\
	& \hspace{30mm} + (M - M( M + \Lambda_t^{22} )^{-1} M)
	\nn\\
	& = M \left[ (M-P_0)^{-1} - ( M + \Lambda_t^{22} )^{-1} \right] M.
	\label{eq:3=>4-eq-1}
\end{align}
Recalling that $\Upsilon(P_0) + Q_t^{22}\in\cSym_{<0}^{n\times n}$ by hypothesis,
\begin{align}
	\Upsilon(P_0) & + Q_t^{22}\in\cSym_{<0}^{n\times n}
	\nn\\
	& \Leftrightarrow (M-P_0)^{-1} - ( M + \Lambda_t^{22} )^{-1}\in\cSym_{<0}^{n\times n}
	\nn\\
	& \Leftrightarrow (M + \Lambda_t^{22}) - (M-P_0) \in\cSym_{<0}^{n\times n}
	\nn\\
	& \Leftrightarrow P_0 + \Lambda_t^{22}\in\cSym_{<0}^{n\times n}\,.
	\label{eq:3=>4-eq-2}
\end{align}
That is, claim {\em 4)} holds.

{\em 4) $\Rightarrow$ 3):} Note that \er{eq:3=>4-eq-1} holds as per the {\em 3) $\Rightarrow$ 4)} case above. By hypothesis, $P_0 + \Lambda_t^{22}\in\cSym_{<0}^{n\times n}$. Hence, the string of equivalences \er{eq:3=>4-eq-2} implies that {\em 3)} holds.

{\em 4) $\Rightarrow$ 5):} Recalling \er{eq:op-G} and \er{eq:kernel-G-and-Q}, the value function $W_t$ of \er{eq:value-W}, \er{eq:W} satisfies
\begin{align}
	& W_t(x)
	= \demi x' P_t x
	= \int_{\R^n}^\oplus G_t(x,y)\otimes \Psi(y)\, dy
	\nn\\
	& = \int_{\R^n}^\oplus \demi \left[ \ba{c}
			x \\ y 
	\ea \right]' \Lambda_t  \left[ \ba{c}
			x \\ y 
	\ea \right] \otimes \demi y' P_0 \, y \, dy
	\nn\\
	& = \demi\int_{\R^n}^\oplus \left[ \ba{c}
							x \\ y
						\ea \right]' \left[ \ba{cc}
							\Lambda_t^{11} & \Lambda_t^{12}
							\\
							(\Lambda_t^{12})' & P_0 + \Lambda_t^{22}
						\ea \right] \left[ \ba{c}
							x \\ y 
						\ea \right] dy
	\nn
\end{align}
for all $x\in\R^n$. By hypothesis, $P_0 + \Lambda_t^{22}\in\cSym_{<0}^{n\times n}$, so that $(P_0 + \Lambda_t^{22})^{-1}$ exists. Hence, the above max-plus integration explicitly evaluates as
\begin{align}
	\demi x' P_t\, x
	& = \demi x' \left[ \Lambda_t^{11} - \Lambda_t^{12} ( P_0 + \Lambda_t^{22} )^{-1} (\Lambda_t^{12})' \right] x.
	\nn
\end{align}
As $x\in\R^n$ is arbitrary, \er{eq:P-and-Lambda} follows immediately. In addition, as {\em 4) $\Leftrightarrow$ 1)}, it immediately follows that 
\begin{align}
	& \sup\left\{ t\in\R_{>0} \left| P_0 + \Lambda_t^{22}\in\cSym_{<0}^{n\times n} \right. \right\}
	\nn\\
	& = \sup\left\{ t\in\R_{>0} \bigl| t\in(0,t^*(P_0)) \right\}
	= 
	t^*(P_0).
	\nn
\end{align}
That is, \er{eq:max-plus-escape} holds.
${ }^{ }$\hfill{\small $\blacksquare$}

\begin{lemma}
\label{lem:monotone}
Given initializations $P_0, \widetilde P_0\in\cSym^{n\times n}$ satisfying $P_0 - \widetilde P_0\in\cSym_{\le 0}^{n\times n}$, the respective unique solutions $P_s, \widetilde P_s\in\cSym^{2n\times 2n}$ of DRE \er{eq:DRE} defined for all $s\in[0,t^*)$, $t^*\doteq\min(t^*(P_0), t^*(\widetilde P_0))$ satisfy
\begin{align}
	& P_s - \widetilde P_s \in\cSym_{\le 0}^{n \times n}
	\label{eq:monotone}
\end{align}
for all $s\in[0,t^*)$.
\end{lemma}
\begin{proof}
Fix $s\in[0,t^*)$. Recalling the notation of the proof of Theorem \ref{thm:Lambda-and-Q}, let $\op{T}:\Delta_{0,t}\rightarrow\R^{n\times n}$ denote the evolution operator associated with the time-dependent ODE
\begin{align}
	\dot Y_\sigma
	& = (\hat A + \demi \hat B \hat B' (P_\sigma + \widetilde P_\sigma))' \, Y_\sigma,
	\nn 
\end{align}
defined for $\sigma\in[0,s]$. In particular, note that
\begin{align}
	\op{T} (\sigma,\sigma)
	& = I\,,
	\nn\\
	\pdtone{}{s}\op{T}(s,\sigma) 
	& = (\hat A + \demi \hat B \hat B' (P_s + \widetilde P_s))' \, \op{T}(s,\sigma) 
	\nn\\
	\pdtone{}{\sigma}\op{T}(s,\sigma) 
	& = - \op{T}(s,\sigma) \, (\hat A + \demi \hat B \hat B' (P_\sigma + \widetilde P_\sigma))'
	\nn	
\end{align}
for all $\sigma\in[0,s]$. Define $\pi:[0,s]\rightarrow\cSym^{n\times n}$ by
\begin{align}
	\pi_\sigma
	& \doteq \op{T}(s,\sigma)\, ( P_\sigma - \widetilde P_\sigma ) \, \op{T}(s,\sigma)'
	\label{eq:monotone-pi}
\end{align}
for all $\sigma\in[0,s]$. Differentiating with respect to $\sigma$,
\begin{align}
	\dot\pi_\sigma & = \ts{\pdtone{}{\sigma}} \op{T}(s,\sigma) \, (P_\sigma - \widetilde P_\sigma) \, \op{T}(s,\sigma)' 
	\nn\\
	& \qquad 
	+ \op{T}(s,\sigma) \, (\dot P_\sigma - \dot{\widetilde P}_\sigma) \, \op{T}(s,\sigma)'
	\nn\\
	& \qquad 
	+ \op{T}(s,\sigma) (P_\sigma - \widetilde P_\sigma) \ts{\pdtone{}{\sigma}} \op{T}(s,\sigma)' 
	\nn\\
	& 
	= \op{T}(s,\sigma) \, \Gamma_\sigma \, \op{T}(s,\sigma)'
	\label{eq:pi-dot}
\end{align}
for all $\sigma\in[0,s]$, where
\begin{align}
	\Gamma_\sigma & \doteq 
	(\dot P_\sigma - \dot{\widetilde P}_\sigma)  - (\hat A + \demi \hat B \hat B' (P_\sigma + \widetilde P_\sigma))' (P_\sigma - \widetilde P_\sigma)
	\nn\\
	& \qquad - (P_\sigma - \widetilde P_\sigma) \, (\hat A + \demi \hat B \hat B' (P_\sigma + \widetilde P_\sigma))
	= 0\,,
	\nn
\end{align}
in which the equality with zero follows by virtue of the fact that $P_\sigma$, $\widetilde P_\sigma$ both satisfy the DRE \er{eq:DRE}. Consequently, \er{eq:pi-dot} implies that $\dot\pi_\sigma = 0$ for all $\sigma\in[0,s]$, so that integration with respect to $\sigma\in[0,s]$ yields $\pi_s = \pi_0$. Recalling \er{eq:monotone-pi}, it follows immediately that 
\begin{align}
	P_s - \widetilde P_s
	& = \pi_s
	= \pi_0 = \op{T}(s,0)\, (P_0 - \widetilde P_0)' \, \op{T}(s,0)'
	\nn
\end{align}
Recalling that $P_0 - \widetilde P_0\in\cSym_{\le 0}^{n\times n}$, and noting that $s\in[0,t^*)$ is arbitrary, yields the required assertion \er{eq:monotone}.
\end{proof}


\bibliographystyle{IEEEtran}
\bibliography{primal}


\end{document}